\documentclass[11pt]{amsart} 
\usepackage{verbatim, latexsym, amssymb, amsmath,color}
\usepackage{epsfig}
\usepackage{tikz}
\usepackage{hyperref}
\usepackage{float}
\usetikzlibrary{matrix}
\usepackage{graphicx}
\def\F{\text{\bf F}}
\def\A{\mathcal{A}}
\def\R{\mathbb R}
\def\N{\mathbb N}
\def\Z{\mathcal Z}
\def\LL{\textbf{L}}
\def\B{\tilde{B}}
\def\S{\tilde{S}}
\def\Atil{\tilde{A}}
\def\Slam{S_{\lambda}}
\def\f{\textbf f}
\def\n{\textbf n}
\def\M{\textbf M}
\def\H{\mathcal H}
\def\V{\mathcal V}
\def\X{\mathfrak{X}}
\def\Flat{\mathcal F}
\def\m{\textbf m}
\def\mga{\Z_n(M, (M\setminus U)\cup \gamma)}
\def\mulam{\mu_{\lambda}}
\def\parj{\frac{\partial}{\partial x_j}}

\def\pars{\frac{\partial}{\partial s}}
\def\pat{\frac{\partial}{\partial t}}
\def\hessq{\text{Hess } r_q}

\newtheorem{thm}{Theorem}[section]
\newtheorem{lemma}[thm]{Lemma}
\newtheorem{claim}[thm]{Claim}
\newtheorem{cor}[thm]{Corollary}
\newtheorem{prop}[thm]{Proposition}
\theoremstyle{remark}
\newtheorem*{rmk}{Remark}

\theoremstyle{definition}
\newtheorem{defi}[thm]{Definition}

\title[Minimal hypersurfaces with fixed boundary]{A mountain pass theorem for minimal hypersurfaces with fixed boundary}
\author{Rafael Montezuma}
\address{Princeton University, Fine Hall, Princeton NJ 08544, USA}
\email{rcabral@math.princeton.edu}

\begin{document}

\begin{abstract}
{ {In this work, we prove the existence of a third embedded minimal hypersurface spanning a closed submanifold $\gamma$ contained in the boundary of a compact Riemannian manifold with convex boundary, when it is known a priori the existence of two strictly stable minimal hypersurfaces that bound $\gamma$. In order to do so, we develop min-max methods similar to those of \cite{dl-r}, by De Lellis and Ramic, adapted to the discrete setting of Almgren and Pitts. }}
\end{abstract}
\maketitle
\setcounter{tocdepth}{1}

\section{Introduction}\label{introduction}


We are concerned with the problem of existence of a third embedded minimal surface spanning a smooth closed curve $\gamma$, possibly with multiple components, when it is known a priori the existence of two other minimal surfaces of minimum type with boundary $\gamma$, or satisfying other natural local minimization condition. We are also interested in the higher dimensional codimension one version of this problem in the Riemannian setting.

Since the classical works of Morse and Tompkins \cite{M+T}, and Shiffman \cite{Sh}, in the parametric setting of the Plateau's problem, this type of problem has been studied extensively in the case of $2$-dimensional submanifolds spanning a contour in Euclidean space. See, for instance, the deep results contained in \cite{JS} and in the references therein.

In a recent paper, \cite{dl-r}, De Lellis and Ramic obtained a very general result of that type. Indeed, our work was motivated by a question posed in \cite{dl-r}. We also apply tools and ideas from that work.

Our main result is the following.

\begin{thm}\label{teorema-A}
Let $M^{n+1}$ be a compact, oriented, Riemannian manifold with strictly convex boundary, and $\gamma^{n-1}$ be a closed, embedded, oriented, smooth submanifold of $\partial M$. Suppose that there exist distinct embedded, oriented, smooth, strictly stable minimal hypersurfaces $\Gamma_1$ and $\Gamma_2$, such that $\partial \Gamma_i = \gamma$, for $i=1,2$, and $\Gamma_1$ and $\Gamma_2$ are homologous. Suppose also that all connected components of each $\Gamma_i$ have non-empty boundary.

Then, there exists a distinct embedded minimal hypersurface $\Sigma$ in $M$, which has a singular set $sing(\Sigma) = \overline{\Sigma}\setminus \Sigma$ of Hausdorff dimension at most $n-7$, and $sing(\Sigma)\cap (\partial M) = \varnothing$. Moreover, $\partial \Sigma = \gamma$ and there is a connected component of $\Sigma$ which is contained neither in $\Gamma_1$, nor in $\Gamma_2$.
\end{thm}

The minimal hypersurface $\Sigma$ obtained in the above theorem could be the union of an embedded closed minimal hypersurface $\Sigma^{\prime}$ that does not intersect $\partial M$ and some connected components of the $\Gamma_i$, $i=1,2$. In this case, $\Sigma^{\prime}$ is disjoint from all these components. As in the proof of Corollary 1.9 of \cite{dl-r}, we are also able to rule out the possibility of $\Sigma$ being simply a combination of connected components of $\Gamma_1$ and $\Gamma_2$. In the case that all connected components of $\Sigma$ have non-empty boundary, this hyperfurface satisfies the further inequality $\H^n(\Sigma)> max\{\H^n(\Sigma_1), \H^n(\Sigma_2)\}$.

If we drop the assumption that the hypersurfaces $\Gamma_i$ are homologous, in the statement of Theorem \ref{teorema-A}, we can minimize area in the homology class of $\Gamma_1-\Gamma_2$ in $H_n(M)$ to obtain a third minimal hypersurface, which is closed.

Our main theorem can be applied when $\gamma$ is embedded in Euclidean space and contained in the boundary of a convex body. More precisely, we have the following consequence of Theorem \ref{teorema-A}.

\begin{cor}\label{teorema-A2}
Let $\gamma$ be a closed, oriented, smooth $(n-1)$-dimensional embedded submanifold of the boundary of a convex body $\Omega$ in $\R^{n+1}$. Suppose that there exist distinct embedded, oriented, smooth, strictly stable minimal hypersurfaces $\Gamma_1$ and $\Gamma_2$ in $\Omega$, with $\partial \Gamma_1 = \partial \Gamma_2 = \gamma$.

Then, there exists a distinct embedded minimal hypersurface $\Sigma$ in $\Omega$, which has a singular set $sing(\Sigma) = \overline{\Sigma}\setminus \Sigma$ of Hausdorff dimension at most $n-7$, with $sing(\Sigma)\subset int(\Omega)$ and $\partial \Sigma = \gamma$. Moreover, there is a connected component of $\Sigma$ which is contained neither in $\Gamma_1$, nor in $\Gamma_2$.
\end{cor}

The simplifications in the statement of Corollary \ref{teorema-A2} are consequences of the maximum principle for minimal hypersurfaces; i.e., there is no immersed closed minimal hypersurface in the Euclidean space $\R^{n+1}$.



In their paper, De Lellis and Ramic proved a result similar to Theorem \ref{teorema-A}, see Corollary 1.9 of \cite{dl-r}. There, the authors were able to obtain the existence of a third minimal hypersurface under an additional hypothesis; they assume that the initial minimal hypersurfaces $\Gamma_1$ and $\Gamma_2$ meet only at the boundary and bound some open region in $M$. In \cite{dl-r}, at the end of Section 1, the authors suggest that a generalization of their Corollary 1.9 in the direction of our main theorem, using a different min-max approach, might be possible.



In order to prove Theorem \ref{teorema-A}, we obtain a min-max existence result for minimal hypersurfaces with fixed boundary. Our approach is based on the original discrete setting of Almgren and Pitts, see \cite{pitts}. The variational methods developed in \cite{pitts} imply that any closed Riemannian manifold, of dimension $3\leq n+1\leq 6$, contains a closed embedded minimal hypersurface. Schoen and Simon \cite{SS} extended this for all dimensions; the closed minimal hypersurface that they obtain in higher dimensional spaces, $n+1\geq 7$, has the same regularity property of that in our main theorem. Recently, mainly after the recent works by Marques and Neves, min-max constructions of minimal submanifolds and its applications have obtained the attention of the mathematical community to a considerable extent.

In the case of closed ambient manifolds, one considers paths in the space $\Z_n(M)$ of codimension one cycles in $M$ starting and ending at the zero cycle. In our setting, we consider $M$ and $\gamma$ as in the statement of Theorem \ref{teorema-A}, and look at paths in the space $\Z_n(M, \gamma)$ of codimension one integral currents in $M$ with boundary supported in $\gamma$. For our main application, we impose further that these paths join the currents induced by the hypersurfaces $\Gamma_i$ that appear in that statement. If $\Pi$ denotes a homotopy class of paths satisfying those properties, we consider the min-max invariant
\begin{equation*}
\LL (\Pi) = \inf_{\{\Sigma_t\}\in \Pi} \sup \{\M(\Sigma_t) : t\in [0,1]\}.
\end{equation*}
Here, $\M(T)$ denotes the mass of the current $T$, which coincides with the $n$-dimensional volume in case $T$ is induced by an oriented hypersurface. 

In the discrete setting, these paths are represented by sequences of maps which are defined in special discrete subsets of the interval $I = [0,1]$, and are finer and finer in the mass norm. We also consider homotopy classes $\Pi \in \Pi_m^{\#}(Z_n(M,\gamma), P_0)$ of $m$-parameter discrete paths in $\Z_n(M, \gamma)$ relative to a fixed map $P_0 : \partial I^m\rightarrow \Z_n(M, \gamma)$, for every positive integer $m$. We use $I^m$ to denote the closed $m$-cube $[0,1]^m = [0,1]\times\cdots\times [0,1]$.

We prove the following min-max existence theorem.

\begin{thm}\label{teorema-B}
Let $M^{n+1}$ and $\gamma^{n-1}$ be as in the statement of Theorem \ref{teorema-A}, and $\Pi \in \Pi_m^{\#}(Z_n(M,\gamma), P_0)$ be a class of sweepouts such that
\begin{equation*}
\LL (\Pi) > \sup\{\M(P_0(x)) : x\in \partial I^m\}.
\end{equation*}
Then, there exists an integral varifold $V\in \V_n(M)$ whose support is the closure of an embedded smooth minimal hypersurface $\Sigma$ such that the closed set $sing(\Sigma) = \overline{\Sigma}\setminus \Sigma$ has Hausdorff dimension at most $n-7$. If $\Sigma_1, \ldots, \Sigma_k$ denote the connected components of $\Sigma$, then:
\begin{enumerate}
\item[(1)] $sing(\Sigma_i)\cap \partial M = \varnothing$, for every $i = 1, \ldots, k$;

\item[(2)] each $\partial \Sigma_i$ is either the empty set or an union of components of $\gamma$. Moreover, $\cup_{i=1}^k \partial \Sigma_i = \gamma$;

\item[(3)] there are positive integers $m_1, \ldots, m_k$, such that $V = \sum_{i=1}^k m_i \overline{\Sigma_i}$. The multiplicity of components $\Sigma_i$ with non-empty $\partial \Sigma_i$ is $m_i=1$;

\item[(4)] $V$ is a limit varifold for some critical sequence in $\Pi$, and
\begin{equation*}
\LL (\Pi) = ||V||(M) =\sum_{i=1}^k m_i \H^n(\Sigma_i).
\end{equation*}
\end{enumerate} 
\end{thm}


In \cite{dl-r}, the authors also obtain a min-max theorem similar to ours. The key difference is the approach that they use to develop their program, which is based on the framework introduced by De Lellis and Tasnady, see \cite{dl-t}. Besides the differences between the two settings, we are still able to apply in our proofs several tools developed in \cite{dl-r}. Among these, we highlight the application of the curvature estimates at the boundary of stable minimal hypersurfaces introduced in Section 6.4 of that paper.


The main advantage of using our approach in the argument of the proof of Theorem \ref{teorema-A} is that the construction of a $1$-parameter family in $\Z_n(M, \gamma)$ joining the hypersurfaces $\Gamma_i$ can be done assuming that these are homologous only. The non-triviality of the homotopy class, expressed by the inequality in the statement of Theorem \ref{teorema-B}, follows from the same arguments as in \cite{dl-r}. Then, we explore this non-triviality to obtain the third minimal hypersurface as an application of our min-max theorem. 

The assumption that all connected components of the $\Gamma_i$ have non-empty boundary, in Theorem \ref{teorema-A}, is used to avoid the following situation. If $\Gamma_1$ has a closed component $\Gamma_1^{\prime}$, the min-max varifold $V$ could be a sum of the components of $\Gamma_1$, with $\Gamma_1^{\prime}$ appearing with multiplicity different from one.

In order to prove Theorem \ref{teorema-B}, we adapt to the present setting several arguments from \cite{C-DL, dl-r, Li-Zhou, MN-willmore, pitts}; including interpolation and discretization results, the pull-tight argument, Pitts' almost minimizing condition and construction of comparison surfaces. As in \cite{dl-r}, the boundary regularity of our comparison surfaces also rely on Allard's theory, see \cite{al-1, al-2}.



In future work, we will study the Morse index of the minimal hypersurface obtained in Theorem \ref{teorema-A}. This will involve a Morse inequality type estimate, see \cite{MN-index} for general Morse index bounds in the closed case.


\subsection*{Acknowledgments} 
The author wishes to express his deep gratitude to professors Fernando Cod\'a Marques and Andr\'e Neves for their support. When this project was started, the author was supported by the EPSRC on the Programme Grant entitled `Singularities of Geometric Partial Differential Equations' reference number EP/K00865X/1.


\subsection*{Organization} The content of this paper is organized as follows:

In Section \ref{minmax-sect}, we fix some basic notation and explain our min-max setting. In Section \ref{sect-interpolation}, we present our adaptations of some technical tools from min-max theory; we describe our interpolation and discretization theorems. In Section \ref{sect-pull-tight}, we recall the stationarity condition introduced in \cite{dl-r}, and include the analog in our setting of the tightening map. In Section \ref{sect-amc}, an almost minimizing condition for varifolds relative to approximating integral currents with a fixed boundary is introduced. We also explain how to adapt to the present framework the main constructions and facts related with that notion. In Section \ref{sect-teoremaB}, we present some tools that are used in the boundary regularity and prove Theorem \ref{teorema-B}. In Section \ref{sect-teoremaA}, we present the remaining arguments to conclude our main result, Theorem \ref{teorema-A}.



\section{GMT notation and min-max definitions}\label{minmax-sect}

In this section, we fix the basic notation that will be used throughout the paper, and present the setting in which we prove the main theorems. In subsection \ref{basic-defs}, we recall some notation from geometric measure theory, and include the definition of the discrete domains of our sweepouts. In subsection \ref{subsect-minmax}, we adapt the basic min-max definitions to our setting.

\subsection{Basic notation}\label{basic-defs}

In this paper we use $M^{n+1}$ and $\gamma^{n-1}$ to denote a compact Riemannian manifold with boundary and a closed submanifold of $\partial M$, respectively. Besides the assumptions in the statement of Theorem \ref{teorema-A}, suppose that $M$ is isometrically embedded in the Euclidean space $\R^N$. The main relevant spaces in this work are: $I_k(M)$ of integral $k$-currents in $\R^N$ which are supported in $M$; the closure $\V_k(M)$ of the space of rectifiable $k$-dimensional varifolds supported in $M$; $\Z_n(M, \gamma)$ of codimension one integral currents $T$ in $M$ with $spt(\partial T)\subset \gamma$. We follow the notation from \cite{MN-willmore, pitts, simon} when we deal with currents, varifolds, and their basic operations. For instance, $\M(T)$ and $|T|$ denote the mass and varifold induced by a current $T$, and $||V||$ denotes the weight of a varifold $V$.

We use $\Flat_M(T-S)$ to denote the flat distance in $M$ between $S, T \in I_k(M)$; more precisely, it is the number defined by
\begin{equation*}
\inf\{\M(P) + \M(Q) : T-S = P+\partial Q, P\in I_k(M) \text{ and }Q\in I_{k+1}(M)\}.
\end{equation*}
If $U\subset M$ is relatively open, and $V, W\in \V_k(M)$, we use $\F_U(V, W)$ to denote the varifold distance between the restrictions of $V$ and $W$ to the Grassmannian of unoriented $k$-planes over $U$, $G_k(U)$, see page 66 of \cite{pitts}. We also use $\F(S,T)$ to denote the $\F$-metric for rectifiable currents; i.e., $\F(S,T) = \Flat_M(T-S) + \F(|S|, |T|)$.

Although we consider maps into a different space of integral currents, $\Z_n(M, \gamma)$, their discrete domains are the same of those in \cite{MN-willmore, pitts}. We follow the usual cell complexes notation from Section 7.1 of \cite{MN-willmore}. Let us briefly recall that notation now. We use $I(1,k)$ to denote the cell complex supported in $I = [0,1]$ whose vertices are the numbers  $0, 3^{-k}, 2\cdot 3^{-k}, \ldots, (3^k-1)\cdot 3^{-k}, 1$, and the $1$-cells are represented by the intervals $[(j-1)\cdot 3^{-k}, j\cdot 3^{-k}]$, for all $j = 1, 2, \ldots, 3^k$. For every $m\in \N$, we consider the product cell complex $I(m,k)$ defined by $I(1,k)^m = I(1,k)\otimes\cdots\otimes I(1,k)$. We also use $I(m,k)_q$ to denote the collection of all $q$-cells of $I(m,k)$.

\subsection{Min-max definitions}\label{subsect-minmax}
In this subsection, we present the setting in which we prove our min-max existence theorem. 

The mass fineness of a discrete map $\varphi : I(m,j)_0 \rightarrow \Z_n(M, \gamma)$, denoted by $\f (\varphi)$ or $\M$-fineness, is defined by the same expression as in 7.2 of \cite{MN-willmore}.

Let us fix $P_0 : \partial I^m \rightarrow \Z_n(M, \gamma)$, a map continuous with respect to the $\F$-metric and such that $\partial P_0(x)= \gamma$, for every $x\in \partial I^m$.

\begin{defi}\label{homotopic}
Let $\varphi_j : I(m,k_j)_0 \rightarrow \Z_n(M, \gamma)$, $j=1,2$, and $\delta>0$. We say that $\varphi_1$ and $\varphi_2$ are $m$-homotopic in $(\Z_n(M, \gamma), P_0)$ with $\M$-fineness $\delta$ if there exists $\Psi : I(1,k)_0 \times I(m,k)_0 \rightarrow \Z_n(M, \gamma)$ with $\f (\Psi) < \delta$, and satisfying
\begin{enumerate}
\item[(a)] $\Psi (0, \cdot) = \varphi_1\circ \n (k,k_1)$ and $\Psi (1, \cdot) = \varphi_2\circ \n (k,k_2)$;

\item[(b)] $\Flat (\Psi(t,x) - P_0(x)) \leq \delta$ and $\M (\Psi(t,x)) \leq \M (P_0(x)) + \delta$, for every $t \in I(1,k)_0$ and $x \in I_0(m,k)_0 = I(m,k)_0 \cap \partial I^m$.
\end{enumerate} 
\end{defi}

The motivation for the above definition is to consider discrete maps that agree with $P_0$ at boundary vertices, and that are finer and finer with respect to the mass norm. Since $P_0$ is continuous in the $\F$-metric only, we allow the maps to assume values that are only close to those of $P_0$. The maps $\n(k, k_j) : I(m,k)_0\rightarrow I(m,k_j)_0$ are the nearest point grid maps, see Section 7.1 of \cite{MN-willmore} for a precise definition.  

\begin{defi}\label{sweepout}
An $(m,\M)$-homotopy sequence of maps into $(Z_n(M,\gamma), P_0)$ is a sequence $\{\varphi_i\}$ of maps from finer and finer $m$-grids $I(m,k_i)_0 = dmn(\varphi_i)$ into $\Z_n(M, \gamma)$, with $k_i$ increasing to $\infty$, such that $\varphi_i$ is $m$-homotopic to $\varphi_{i+1}$ in $(\Z_n(M,\gamma),P_0)$ with $\M$-fineness tending to zero, and such that there exists $C = C(\{\varphi_i\})>0$ for which $\M(\varphi_i(x))\leq C$, for all $i\in \N$ and $x \in dmn(\varphi_i)$. 
\end{defi} 

We observe that if $\{\varphi_i\}$ is an $(m,\M)$-homotopy sequence of maps into $(Z_n(M,\gamma), P_0)$, then $\partial \varphi_i(x)=\gamma$, for large $i$ and every $x$ in the domain of $\varphi_i$. In order to justify this claim, we apply the above definitions to obtain a sequence $\delta_i \searrow 0$ for which $\f(\varphi_i) \leq \delta_i$, and $\Flat (\varphi_i(x) - P_0(x)) \leq \delta_i$, for all $x$ in $(\partial I^m) \cap dmn(\varphi_i)$. The end of the argument consists of two straightforward applications of the following observation.

\begin{lemma}\label{lemm.bdry.gamma}
There exists $\delta >0$ with the following property: if $S,T \in \Z_n(M,\gamma)$ are such that $\partial S = \gamma$ and $\Flat_M (T-S)< \delta$, then $\partial T = \gamma$. 
\end{lemma}

\begin{proof}
Suppose, by contradiction, that there is no such $\delta >0$. Then, we can find sequences $\{S_i\}_{i\in \N}$, $\{T_i\}_{i\in \N}$, with $\partial S_i = \gamma$ and $\partial T_i \neq \gamma$, for all $i \in \N$, and $\Flat_M(T_i-S_i) \rightarrow 0$. Let $R_i = \partial (T_i-S_i) = \partial T_i - \gamma$. Since the $(n-1)$-boundary $R_i$ is supported on the $(n-1)$-dimensional closed, oriented, submanifold $\gamma$, we know that each $R_i$ is the rectifiable current induced by some linear combination of the connected components of $\gamma$ with integer coefficients. This is an application of the Constancy Theorem. 

Evaluating $R_i$ at the volume form of any component of $\gamma$, we obtain an integer multiple of the $(n-1)$-volume of this component. Using that $\Flat_M(S_i-T_i)\rightarrow 0$ implies $\Flat_M(R_i)\rightarrow 0$, and that the latter implies that $R_i$ weakly converges to the zero current, we conclude that the coefficients of $R_i$ vanish for large $i\in \N$. In conclusion, $\partial T_i - \gamma = R_i = 0$, for large $i$, which is a contradiction. This proves the lemma.
\end{proof}

\begin{rmk}
The assumption that $\partial S = \gamma$ in the statement of Lemma \ref{lemm.bdry.gamma} can be relaxed. The same argument proves the existence of a positive $\delta$ for which: if $S,T \in \Z_n(M,\gamma)$ are such that $\Flat_M (T-S)< \delta$, then $\partial S = \partial T$. 
\end{rmk}

In the present setting, we observe a property of the boundary values of $\{\varphi_i\}$, an $(m,\M)$-homotopy sequence of maps into $(Z_n(M,\gamma), P_0)$, analogous to that noted in Lemma 7.8 of \cite{MN-willmore} in the case of sweepouts by cycles. Namely, if we use $dmn(\varphi_i)$ to denote the domain of the map $\varphi_i$, then
\begin{equation}\label{brdy-values-sweepouts}
\lim_{i\rightarrow \infty} \sup\{\F(\varphi_i(x), P_0(x)): x \in dmn(\varphi_i)\cap \partial I^m\} = 0.
\end{equation}
This property follows from the fact that $P_0$ is continuous in the $\F$-metric topology, item (b) in definition \ref{homotopic}, and Lemma \ref{lemma4.1ofMN}.

Similarly to Definition 7.9. in \cite{MN-willmore}, we consider an equivalence relation in the set of $(m,\M)$-homotopy sequences of maps into $(Z_n(M,\gamma), P_0)$, and use $\Pi_m^{\#}(Z_n(M,\gamma), P_0)$ to denote the set of all equivalence classes.

Consider $\Pi \in \Pi_m^{\#}(Z_n(M,\gamma), P_0)$. For every $\{\varphi_i\}_{i\in \N} \in \Pi$, we consider
\begin{equation}
L(\{\varphi_i\}) = \limsup_{i\rightarrow \infty} max\{\M(\varphi_i(x))\},
\end{equation}
where the maximum is taken over all $x$ in the domain of $\varphi_i$. Then, we use these quantities to define the min-max invariant of the class $\Pi$, namely
\begin{equation}
\LL (\Pi) = \inf\{ L(\{\varphi_i\}) : \{\varphi_i\} \in \Pi\}.
\end{equation}
Similarly to the definitions in 7.12 of \cite{MN-willmore}, we consider the compact subsets $K(\{\varphi_i\})$ of limit varifolds $V = \lim_j |\varphi_{i_j}(x_j)|$ with $i_j \rightarrow \infty$, and, in case $\{\varphi_i\}$ is a critical sequence (i.e., $L(\{\varphi_i\}) = \LL (\Pi)$), the critical set 
\begin{equation*}
Crit(\{\varphi_i\}) = K(\{\varphi_i\})\cap \{V : ||V||(M) = \LL (\Pi)\}.
\end{equation*}
Critical sequences always exist. The verification of this fact is analogous to the proof of Lemma 15.1 of \cite{MN-willmore}.

\section{Interpolation results}\label{sect-interpolation}

In this section, we adapt to the present setting the interpolation results of the min-max theory of Almgren and Pitts, which are key technical tools of the subject. In the first two subsection, \ref{interpolation1} and \ref{Sec.discretization}, we present version for currents with non-vanishing fixed boundary of the main theorems of sections 13 and 14 of \cite{MN-willmore}. In subsection \ref{sect-local-interpolation}, we explain an extension of the local interpolation lemma, Lemma 3.8 of \cite{pitts}. The latter requires more work, because we need to consider a situation in which concentration of mass is possible. This difficulty is discussed in details in the beginning of \ref{sect-local-interpolation}.



\subsection{Discrete to continuous}\label{interpolation1}

Let us start by recalling the interpolation from discrete to continuous maps. In his paper on homotopy groups of integral cycle groups, \cite{alm1}, Almgren presented a version of such interpolation result. That version was improved by Pitts \cite{pitts} to be applied in his general existence theorem, and later by Marques and Neves \cite{MN-willmore} in their proof of the Willmore Conjecture.

In what follows, we use $\delta_0$ to denote the positive constant which allows us to apply Corollary 1.14 of \cite{alm1}; see also Section 14 of \cite{MN-willmore}.

\begin{thm}\label{interpolation}
Let $M^{n+1}$ be a closed Riemannian manifold. There exists $C = C(M,m)>0$ such that: any discrete map $\psi : I(m,0)_0 \rightarrow \Z_n(M)$ with $\M$-fineness $\f (\psi)< \delta_0$, can be extended to a map $\tilde{\psi} : I^m \rightarrow \Z_n(M;\M)$ continuous in the mass norm. Moreover, we have that:
\begin{itemize}
\item the values of $\psi$ on the vertices of a $p$-cell $\alpha$ determine $\tilde{\psi}_{|\alpha}$;

\item if $Q(\alpha)$ denotes a $\Flat$-isoperimetric choice for $\partial \alpha$, $\alpha \in I(m,0)_1$, then
\begin{equation*}
\M(\tilde{\psi}(x) - \tilde{\psi}(y)) \leq C \sup \{\M(\partial Q(\alpha))\},
\end{equation*}
where the supremum is taken over all $1$-cells of $I(m,0)$.
\end{itemize}
\end{thm}

The version of this theorem in Almgren's paper is weaker than this one because it extends discrete maps to maps that are continuous in the weaker flat topology. 
Almgren's result is obtained as a combination of Theorem 2.4 and Section 6.5 of \cite{alm1}. On the other hand, observe that Almgren states his theorems for maps with values in spaces of currents that are slightly more general, the spaces $\Z_k(A,B)$ of integral currents in $A$ with boundary supported in $B$. Using $A = M$, a manifold with boundary, and $B = \gamma^{n-1}$, a submanifold of $\partial M$, we obtain an interpolation result for currents with boundary $\gamma$ similar to Theorem \ref{interpolation}.

Before stating the theorem, we observe that the number $\delta_0 = \delta_0(M, \gamma)>0$ that appears in the statement, analogous to the constant $\delta_0$ of Theorem \ref{interpolation}, is the one given by Theorem 2.4 of \cite{alm1}; i.e., $\nu_{A,B}$ for our choices of $A$ and $B$.

\begin{thm}\label{interpolation_boundary}
Let $M$ be a compact Riemannian $(n+1)$-manifold with boundary, and $\gamma$ be an embedded closed oriented hypersurface of $\partial M$. There exists $C >0$ such that: any discrete map $\psi : I(m,0)_0 \rightarrow \Z_n(M,\gamma)$ with $\M$-fineness $\f (\psi)< \delta_0$, and constant boundary values
\begin{equation*}
\partial \psi (x) = \gamma, \text{ for all } x \in I(m,0)_0,
\end{equation*}
 admits a continuous extension to a map $\tilde{\psi} : I^m \rightarrow \Z_n(M,\gamma;\M)$ from the $m$-dimensional cube and with constant boundary values ($\partial \tilde{\psi} (\cdot) = \gamma$). 
\end{thm}

\begin{rmk}
The properties of $\tilde{\psi}$ mentioned at the end of Theorem \ref{interpolation} also hold for the extension obtained on Theorem \ref{interpolation_boundary}. 
\end{rmk}

Theorem \ref{interpolation_boundary} is a combination of Theorem \ref{interpolation} and Almgren's version of that result. Actually, their combination implies the existence of an extension of $\psi$ to a continuous map $\tilde{\psi} : I^m \rightarrow \Z_n(M,\gamma;\M)$. In particular, it is already known that the support of the boundary of $\tilde{\psi}(x)$ is contained in $\gamma$, for all $x \in I^m$. It remains only to verify that the above inclusion is an equality. But this is a straightforward application of Lemma \ref{lemm.bdry.gamma}.

Summarizing, Almgren's version of this type of interpolation takes care of the necessary operations involving the boundary and of the extension, continuity with respect to the mass norm follows the same steps as in the argument of Marques and Neves, and the Constancy Theorem is applied, via Lemma \ref{lemm.bdry.gamma}, to assure that all currents that appear have boundary $\gamma$.


\subsection{Discretization of maps}\label{Sec.discretization}

The second interpolation theorem that we review and extend is used to produce fine discrete maps corresponding to maps that are continuous with respect to the flat norm topology. 

Let $\tilde{M}$ be a closed $(n+1)$-dimensional submanifold of $\R^N$, $A$ be a compact Lipschitz neighborhood retract, abbreviated CLNR, subset of $\subset \tilde{M}$, and $B\subset A$ be a closed, embedded, smooth submanifold of $\tilde{M}$. For a precise definition of CLNR sets see Definition 1.1 of \cite{alm1}. 

In what follows, assume that $A=M$ is a compact $(n+1)$-dimensional manifold with boundary contained in $\tilde{M}$, and $B = \gamma^{n-1}$ is a submanifold of $\partial M$, as in the statement of Theorem \ref{teorema-A}. Let
\begin{equation*}
\phi : I^m \rightarrow \Z_{n}(A,B)
\end{equation*}
be continuous in the flat norm, $\Flat = \Flat_{A}$, and satisfying the following:
\begin{enumerate}
\item[(i)] restricted to $\partial I^m$, $\phi$ is continuous with respect to the $\F$-metric;

\item[(ii)] $ \sup \{\M(\phi(x)) : x \in I^m\} < \infty$, and
\item[(iii)] $\limsup_{r\rightarrow 0} \m(\phi,r) = 0$,
\end{enumerate}
where $\m(\phi,r) = \sup \{||\phi(x)||B_r(p) : p\in A \text{ and } x \in I^m\}$, and $B_r(p)$ are the geodesic open balls in the closed manifold $\tilde{M}$. Under conditions (i), (ii) and (iii), we have the following discretization result.

\begin{thm}\label{discretization}
There exist $\delta_i \searrow 0$, $\ell_i \nearrow +\infty$, $\varphi_i : I(m,k_i)_0 \rightarrow \Z_n(A,B)$, and $\Psi_i : I(1,k_i)_0\times I(m,k_i)_0 \rightarrow \Z_n(A,B)$, with $k_i \nearrow +\infty$ and such that
\begin{enumerate}
\item[(a)] $\Psi_i(0,x) = \varphi_i(x)$ and $\Psi_i(1,x) = \varphi_{i+1}(x)$, for all $x\in I(m,k_i)_0$;

\item[(b)] the $\M$-fineness of $\Psi_i$ satisfies $\f(\Psi_i) < \delta_i$;

\item[(c)] $\M (\varphi_i(y)) \leq \sup \{\M(\phi(x)) : x,y \in \alpha, \alpha \in I(m,\ell_i)_m \} + \delta_i$, for every $y$ in $I(m,k_i)_0$;

\item[(d)] $\M(\Psi_i(t,x)) \leq  \M(\phi(x)) + \delta_i$, for all $(t,x)$ in $I(1,k_i)_0 \times I_0(m,k_i)_0$;

\item[(e)] $\Flat (\Psi_i(t,x) - \phi (x)) \leq \delta_i$, for all $(t,x)$ in $I(1,k_i)_0 \times I(m,k_i)_0$.
\end{enumerate}
Moreover, if $\phi$ is continuous with respect to $\M$ on an $(m-1)$-face of $I^m$, then the $\varphi_i$ can be chosen such that they coincide with $\phi$ on that face. 
\end{thm}

The above result is the analog of Theorem 13.1 of \cite{MN-willmore} in our setting. There the authors consider maps taking values into the space of integral cycles supported on a closed manifold. Their proof can be adapted to justify the above statement in a straightforward manner. The main observation is that a manifestation of the Constancy Theorem, similar to that of Lemma \ref{lemm.bdry.gamma}, implies that besides the fact that the currents that we started with have boundary, $\Flat_A$-isoperimetric choices can be done as usual. The cut and paste construction using slicing theory can be done using distance functions on $\tilde{M}$. Moreover, estimates of $\Flat_{\tilde{M}}$ regarding the currents in question can be used as $\Flat_A$ estimates, because all currents that appear are supported in $A$. 

Next, we include a brief discussion on the proof of Theorem \ref{discretization} with some further details about the facts mentioned in the previous paragraph. As it was also remarked by other authors, see \cite{Li-Zhou}, the only step of the proof where Marques and Neves do explicit operations on integral cycles appears in Lemma 13.4 of \cite{MN-willmore}. There, they fix a current, that in our case is $T \in \Z_n(A,B)$, and consider maps from a fixed finite set with image into
\begin{equation*}
{\textbf B}^{\Flat}_{\varepsilon_j} (T) \cap \{S : \M(S)\leq 2L \},
\end{equation*}
for a sequence $\varepsilon_j \searrow 0$, where ${\textbf B}^{\Flat}_{r} (T)$ denotes the ball in $\Z_n(A,B)$ of radius $r$ with respect to the flat norm topology and center at $T$, and $L>0$ is fixed. Since the compact set $A$ is contained in $\tilde{M}$, we have that $\Flat_{\tilde{M}}\leq \Flat_A$. In particular,
convergence in $\Flat_A$ implies convergence in $\Flat_{\tilde{M}}$, which implies convergence as currents. Varifold limits of sequences of induced varifolds $|S_j|$, where $S_j \in {\textbf B}^{\Flat}_{\varepsilon_j} (T)$, are also supported on $A$. Then, the finite sequences $\{p_i\}$ and $\{r_i\}$ that appear in \cite{MN-willmore} page 748 can be chosen such that $p_i \in A$ and $B_{r_i}(p_i)$ are geodesic open balls in $\tilde{M}$, for every $i$. 

The $\Flat$-isoperimetric choices can still be performed using the same result, Corollary 1.14 of \cite{alm1}. Indeed, we observe that: if $S \in {\textbf B}^{\Flat}_{\varepsilon_j} (T)$, then $\partial S = \partial T$, for large $j\in \N$. This is a straightforward application of the Constancy Theorem, similar to that of Lemma \ref{lemm.bdry.gamma}. Hence, $\partial (S-T) = 0$, which implies that we can find an $\Flat_A$-isoperimetric choice $Q = Q(S-T) \in I_{n+1}(A)$ for $S-T \in \Z_n(A)$. The cutting operations that Marques and Neves do in the proof of Theorem 13.1, \cite{MN-willmore}, using slicing theory can be done in the present setting using distance functions of the closed manifold $\tilde{M}$. More precisely, they consider slices of isoperimetric choices, such as the current $Q$ above, and these are $(n+1)$-dimensional and supported on $A$.  It follows, by definition, that the supports of those slices are also contained in $A$. This explains why we can still perform the cut and paste argument to prove the analog of Lemma 13.4 and Proposition 13.3 from \cite{MN-willmore}. The remaining steps are successive applications of this proposition. Hence, Theorem \ref{discretization} follows.


\subsection{Local interpolation}\label{sect-local-interpolation}

In this subsection, $U \subset M$ denotes a relatively open subset and $C\subset M$ a compact set such that $U\subset C$. The following lemma plays a crucial role in the proof of the combinatorial argument. It is the analog in our setting to Lemma 3.8 in \cite{pitts}, and Lemma B.1 in \cite{Li-Zhou}.

\begin{lemma}\label{interp.flat-to-mass}
Consider $L>0$, $\delta >0$, $U\subset M$ relatively open, $K\subset U$ compact, and $T \in \Z_n(M, (M\setminus U)\cup \gamma)$ with $(\partial T)\llcorner U  = [|\gamma|]\llcorner U$. There exists $\varepsilon >0$ such that whenever
\begin{itemize}
\item $S_1, S_2 \in \mga$
\item $(\partial S_i) \llcorner U  = [|\gamma|] \llcorner U , i=1,2$
\item $\Flat_C(S_1-S_2) \leq \varepsilon$
\item $spt(S_1-T)\cup spt(S_2-T)\subset K$
\item $\M(S_1)\leq L, \M(S_2)\leq L$,
\end{itemize}
there is a finite sequence $T_0, T_1, \ldots, T_l \in \mga$ with $T_0 = S_1$, $T_l = S_2$, $(\partial T_j)\llcorner U = [|\gamma|]\llcorner U$, for every $j=1, \ldots, l$, and
\begin{equation*}
\bigcup_j spt(T_j-T)\subset U, \ \ \sup_j \M(T_j-T_{j-1})\leq \delta, \text{ and } \sup_j \M(T_j) \leq L+\delta. 
\end{equation*}
\end{lemma}

The proof of Lemma \ref{interp.flat-to-mass} is similar to the arguments of Lemma 3.8 in \cite{pitts}, and Lemma B.1 in \cite{Li-Zhou}. As in those cases, we observe that it suffices to prove the existence of such an $\varepsilon>0$ in the case that one of the currents $S_i$ is fixed. In other words, if we consider $S_1\in \mga$, with $(\partial S_1) \llcorner U  = [|\gamma|] \llcorner U$, $spt(S_1-T)\subset K$, and $\M(S_1)\leq L$, it is enough to prove the existence of an $\varepsilon>0$, possibly depending on $S_1$, for which the result holds. Indeed, this is enough to prove the lemma because the set of all such currents $S_1$ is compact with respect to the topology induced by $\Flat_C$.

If there is a finite sequence $T_0=S_1, T_1, \ldots, T_l=S$ as in the statement of the lemma, we say that one can interpolate between $S_1$ and $S$. For fixed $S_1$, the proof is an argument by contradiction. Consider a sequence $R_i \in \mga$, $i \in \N$, with $\Flat_C(R_i-S_1) \rightarrow 0$, satisfying the required boundary, support and mass constraints, and such that it is not possible to interpolate between $S_1$ and $R_i$. 

In \cite{pitts}, proof of Lemma 3.8, the author's contradiction argument assumes that the induced varifolds $|R_i|$ converge to $|S_1|$. That was possible because of his Lemma 3.7. Proving that this is also the case in our setting is the part of the argument that deserves more attention. Indeed, once we know that $|R_i|$ converges to $|S_1|$ as varifolds, it follows from the rectifiability of $|S_1|$ that there is no concentration of mass over this sequence. Therefore, the cut and paste argument performed in the rest of the proof of Lemma 3.8 of \cite{pitts}, also in Lemma 13.4 of \cite{MN-willmore}, can be similarly carried out as it was briefly mentioned after the statement of our Theorem \ref{discretization}. 

In the present argument, it is even easier to verify that one can consider the usual $\Flat$-isoperimetric choice for $R_i - S_1$ than it was in Section \ref{Sec.discretization}. Indeed, as in that case, it suffices  to check that $R_i-S_1$ is a cycle. Observe that $\partial (R_i -S_1)= \partial (R_i -T)+\partial (T-S_1)$ is supported in $K$, then
\begin{equation*}
\partial (R_i -S_1) = (\partial R_i -\partial S_1)\llcorner U = (\partial R_i)\llcorner U - (\partial S_1)\llcorner U =0,
\end{equation*}
since both terms in the last expression before the zero are equal to $[|\gamma|]\llcorner U$.

Then, it remains to adapt the interpolation arguments of Pitts' Lemma 3.7. Assume that $|R_i|$ converges to a varifold $V \in \V_n(M)$. The proof is divided into two cases based on the existence of points $q\in U\cap spt(||V||)$ for which $||V||\{q\} > \alpha = \delta/5$, where $\delta >0$ is part of the initial data.

In the first of these two cases, where one assumes that no such points exist, the proof is analogous to the usual cut and paste argument as described above. Therefore, its adaptation follows as before.

In the second case, we have that $|R_i|$ converges to $V$ and $||V||\{q\} > \alpha$, for some $q\in U\cap spt(||V||)$. In Lemma 3.7 of \cite{pitts}, this case is reduced to the first after an appropriate choice of geodesic balls $B_{r_i}(p_i)$, with $q \in B_{r_i}(p_i)$, $p_i \rightarrow q$, and $r_i\rightarrow 0$, and explicitly constructed interpolations between the $R_i$ and the currents given by
\begin{eqnarray*}
(R_i)\llcorner (M \setminus B_{r_i}(p_i)) + (exp_{p_i})_{\#}\left(\delta_0 \mathbb{X} (exp_{p_i}^{-1})_{\#}(\partial(R_i\llcorner B_{r_i}(p_i)))\right), 
\end{eqnarray*}
where $exp_{p_i}$ denotes the usual exponential map at $p_i$ and $\delta_0 \mathbb{X} S$ denotes the cone over $S$ in the Euclidean space, see topic 26.26 of \cite{simon}.

Our main difficulty was to deal with the construction of these conical replacements in the case $q \in \gamma$, and the corresponding interpolations. In other words, we need a result analogous to Lemma 3.5 in \cite{pitts}. This is the main missing ingredient at this point, since the choice of the domains $\tilde{B}_{r_i}(p_i)$ where we switch $R_i$ by its conical replacement follows from almost the same steps as in Pitts' argument in \cite{pitts} page 119. The only step of that part that remains to be adapted is where Lemma 3.6 of \cite{simon} is applied. These two issues are settled in our Lemmas \ref{conical-replacements} and \ref{choice-p_i}, respectively.

For obvious reasons we can not use the same construction at points $p_i \in \gamma$. Li and Zhou \cite{Li-Zhou} dealt with a similar difficulty in the case of relative cycles. There, they used Fermi coordinates to describe a neighborhood of $p_i \in \partial M$. Since these coordinates take values on a half-space, or even in a half-ball centered at the origin if one chooses the right neighborhood, it is possible to perform the cone construction in the Euclidean space and push it back to the manifold. After a delicate adaptation of Pitts' methods from 3.4-3.6, Li and Zhou succeeded in interpolating between the initial currents $R_i$ and
\begin{equation*}
(R_i)\llcorner (M \setminus \hat{B}_{r_i}(p_i)) + (\widehat{exp}_{p_i})_{\#}\left(\delta_0 \mathbb{X} (\widehat{exp}_{p_i}^{-1})_{\#}(\partial_1(R_i\llcorner \hat{B}_{r_i}(p_i)))\right),
\end{equation*}
where $\widehat{exp}_{p_i} = (x,t)$ denotes Fermi coordinates at $p_i \in \partial M$, $x = (x_1, \ldots, x_n)$ represents geodesic normal coordinates of $\partial M$ at $p$, and the coordinate $t$ is the distance to $\partial M$ in $M$. The domain $\hat{B}_{r_i}(p_i)$ is the sub-level $\{\hat{r}_{p_i}< r_i\}$ of the function $\hat{r}_{p_i} = \sqrt{|x|^2+t^2}$, and $\partial_1(R_i\llcorner \hat{B}_{r_i}(p_i))$ denotes the restriction of $\partial (R_i\llcorner \hat{B}_{r_i}(p_i))$ to the level set $\{\hat{r}_{p_i}=r_i\}$. See Lemma B.8 in \cite{Li-Zhou}.

Unfortunately, we can not use Li-Zhou's construction of conical replacements because Fermi coordinates do not preserve our boundary constraint $(\partial R_i)\llcorner U = [|\gamma|]\llcorner U$. 

To fix that issue, we introduce the following system of coordinates near $p\in \gamma$. We use $E_p = (x,s,t)$, where $x=(x_1, \ldots, x_{n-1})$ are geodesic normal coordinates for $\gamma$ at $p$, the coordinate $t$ represents distance to $\partial M$ in $M$, and $s=s(y)$ is a signed distance function to $\gamma$ in $\partial M$ applied to $\pi_{\partial M}(y)$. Here, we use $\pi_{\partial M}$ to denote the nearest point projection to $\partial M$ of points $y\in M$ near $p$. If $N(x)$ denotes the unit normal vector to $\gamma$ in $\partial M$ at the point of coordinates $(x,0,0)$, and $exp_{x}^{\partial M}$ is the usual exponential map of $\partial M$ at $(x,0,0)$, then $exp_{x}^{\partial M}(sN(x))$ is the point of coordinates $(x,s,0)$. 

Letting $\nu(x,s)$ denote the inward pointing unit normal to $\partial M$ at $(x,s,0)$, and $exp_{(x,s)}^M$ the usual exponential map of $M$ at $(x,s,0)$, one can easily check that the coordinates $(x,s,t)$ describes $exp_{(x,s)}^M (t \nu(x,s))$.

We use $E_p(x,s,t)$ to denote the point in $M$ with coordinates $(x,s,t)$. Using the canonical identification between $T_pM$ and the Euclidean half-space $\R^{n+1}_{+}$, we can assume that $E_p$ is defined on a neighborhood of the origin $0_p$ of the half-space $T_pM$, and takes values in $M$. As in the case of the usual exponential map, we have that $D(E_p) (0_p)$ equals the identity of $T_pM$. Then, it follows that $E_p$ is really a local diffeomorphism near $0_p$.

Consider the function $r_p(x,s,t) = \sqrt{|x|^2+s^2+t^2}$, defined on the domain parametrized by $E_p$, where $|x|^2 = x_1^2+\ldots +x_{n-1}^2$. We use the following notation for sub-level and level sets of this function: 
\begin{equation*}
\B_{r}(p) = \{r_p(x,s,t) < r\} \text{ and } \S_{r}(p) = \{r_p(x,s,t) = r\}.
\end{equation*}

We are now ready to state the version of the interpolation result to conical replacements adapted to the fixed boundary setting.

\begin{lemma}\label{conical-replacements}
For every $q \in \gamma\cap U$ and $0< \varepsilon <1$ there exists a relatively open set $Z$, with $q\in Z \subset U$, such that: given

\begin{enumerate}
\item[(a)] $p\in \gamma \cap Z$ and $r>0$, such that $\B_r(p)\subset\subset Z$

\item[(b)] $\delta >0$ and $T \in \Z_n(clos(\B_r(p)), \gamma \cup \S_r(p))$, such that
\begin{eqnarray*}
\left.\begin{aligned} 
\nonumber & (\partial T) \llcorner \B_r(p) = [|\gamma |] \llcorner \B_r(p) \\
\nonumber & 2 \M (  (\partial T)\llcorner \S_r(p) ) rn^{-1} + \delta \leq \M (T) \\
\nonumber & \varepsilon \M (T) \leq n \delta \\
\nonumber & ||T||(\S _{\rho}(p)) = 0, \text{ for all } 0 \leq \rho \leq r.
\end{aligned}\right.
\end{eqnarray*}
\end{enumerate}
Then, for every $\beta >0$ there exists a finite sequence $T=P_0, P_1,\ldots, P_m $ in $\Z_n(clos(\B_r(p)), \gamma \cup \S_r(p))$ such that
\begin{eqnarray*}
\left.\begin{aligned} 
\nonumber & \partial P_j = \partial T \\
\nonumber & \M (P_j) \leq \M(T) + \beta \\
\nonumber & \M(P_j - P_{j-1}) \leq \beta \\
\nonumber & P_m = (E_p)_{\#} \left( \delta_0 \mathbb{X} (E_p^{-1})_{\#}((\partial T)\llcorner \S_r(p))\right)\\
\nonumber &\M(P_m) \leq 2r^{-1} \M ((\partial T)\llcorner \S_r(p)) \leq \M(T) - \delta.
\end{aligned}\right.
\end{eqnarray*}
\end{lemma}

And the next lemma settles the second issue mentioned above.

\begin{lemma}\label{choice-p_i}
Let $Z$ be a region as in Lemmas \ref{conical-replacements} and \ref{choice-Z}, and $V \in \V_n(M)$, which is rectifiable on $Z$. Then, for $\H^{n-1}$-almost all $p\in \gamma\cap Z$, we have that $||V||\S_{\rho}(p)=0$, whenever $0<\rho$ is such that $\B_{\rho}(p)\subset\subset Z$. 
\end{lemma}

We present the proof of Lemmas \ref{conical-replacements} and \ref{choice-p_i} at the end of this section. In what follows, we introduce some preliminary observations.

First of all, we observe that making the coordinated neighborhood sufficiently small we can assume that the metric coefficients with respect to $E_q$ are such that $g_{tt}=1$, $g_{it} = g_{st}=0$, and $|g_{ij}-\delta_{ij}|$, $|g_{ss}-1|$, and $|g_{is}|$ are arbitrarily small. Moreover, we can assume that the Christoffel symbols are such that, for all $1\leq i,j,k \leq n-1$ and $a,b \in \{1,\ldots,n-1\}\cup\{s\}$, 
\begin{eqnarray*}
\left.\begin{aligned} 
\nonumber & \Gamma_{tt}^t = \Gamma_{at}^t = \Gamma_{tt}^a = 0\\
\nonumber & \Gamma_{ij}^k, \Gamma_{ss}^a, \Gamma_{is}^s \text{ are small}\\
\nonumber & \Gamma_{ij}^s, \Gamma_{is}^j \text{ are uniformly bounded in terms of } II_{\gamma\subset \partial M}\\
\nonumber & \Gamma_{at}^b, \Gamma_{ab}^t \text{ are uniformly bounded in terms of } II_{\partial M\subset M},
\end{aligned}\right.
\end{eqnarray*}
where $II_{\gamma\subset \partial M}$ and $II_{\partial M\subset M}$ denote the second fundamental forms of the immersions of $\gamma$ in $\partial M$, and of $\partial M$ in $M$, respectively. 

Since the derivatives of the metric coefficients can be expressed as a linear combination of the Christoffel symbols with the metric coefficients itself, we see that the Lipschitz constants $\text{Lip}(g_{ab})$, $a,b \in \{1,\ldots, n-1\}\cup\{s,t\}$, can be bounded independently of $q$. This can extended as an uniform bound for the Lipschitz constants of the higher dimensional metric tensors
\begin{equation*}
g_k(y)(v,w) = \langle v, \Lambda_k D(E_q)(y) \rangle \cdot \langle w, \Lambda_k D(E_q)(y) \rangle,
\end{equation*}
where $y \in T_qM$, $v,w \in \Lambda_k(T_qM)$ are $k$-vectors, and $\langle \cdot, \cdot \rangle$ is the canonical pairing between vectors and covectors, for $1\leq k \leq n+1$. See paragraph 25 in \cite{simon} for the basic notation on vectors and covectors.

Next, we observe that it follows from $D(E_q)(0_q) = Id$, the fact that we can assume that
\begin{equation*}
||\Lambda_n D (E_{q}^{-1})||, \ \ \text{Lip}(E_q|_{B^+_{R}(0)}), \text{ and } \text{Lip}(E_q^{-1}|_{\B_R}(q))
\end{equation*} 
are uniformly, with respect to $q$, close to $1$. We use $B^+_{R}(0)$ to denote the Euclidean ball of radius $R$ centered at the origin $0_q \in T_qM$.

Let $0\leq \lambda \leq 1$, and consider the map $\mu_{\lambda} : T_qM \rightarrow T_qM$ defined by $\mu_{\lambda}(v) = \lambda v$. The following claim is a consequence of our previous estimates.

\begin{claim}\label{derivative-mult-by-lambda}
Given $\varepsilon >0$, we have that
\begin{equation*}
||\Lambda_n D (E_q \circ \mulam \circ E_{q}^{-1})(y)|| \leq \lambda^n (1+\varepsilon (1-\lambda)),
\end{equation*}
whenever $\lambda \in[0,1]$, and $y$ is sufficiently close to $q$, possibly depending on $\varepsilon$.
\end{claim}

The proof of this fact is analogous to the argument of the equivalent  inequality in Section 3.4(3) of \cite{pitts}. This is true because an uniform bound on the Lipschitz constant of the metric coefficients is enough to go through all steps of that argument. This was also noted in Lemma B.4 of \cite{Li-Zhou}.

Let us prove now that the domains $\B_r(q)$ are relatively convex in the sense that the second fundamental form of $\S_r(q)\subset M$ with respect to the unit normal vector that points to the interior of $\B_r(q)$ is positive definite. Observe that
\begin{equation}\label{eq-gradient}
\nabla r_q = \left( g^{ij} \cdot \frac{x_i}{r_q} + g^{sj} \cdot\frac{s}{r_q}\right) \parj + \left( g^{is} \cdot\frac{x_i}{r_q} + g^{ss}\cdot \frac{s}{r_q}\right) \pars + \frac{t}{r_q} \pat.
\end{equation}
We know that $\nabla r_q$ is normal to the $\S_r(q)$. Moreover, it follows from previous steps that we can assume that $|\nabla r_q|$ is nearly $1$, so it is close to be the unit normal vector of the hypersurfaces $\S_r(q)$. Observe that
\begin{equation*}
\hessq (v,w) = g(\nabla_{v} \nabla r_q, w) = - g(\nabla r_q, II_{\S_r(q)}(v,w)),
\end{equation*}
for $v,w \in T_y\S_r(q)$, where $ II_{\S_r(q)}$ denotes the second fundamental form of $\S_r(q)\subset M$. Therefore, $\hessq (v,w)$ approximates the scalar value of $ II_{\S_r(q)}(v,w)$ with respect to the inward pointing unit normal.

A direct computation using (\ref{eq-gradient}) yields $\hessq (\partial_t,\partial_t) = r_q^{-1} - t^2 r_q^{-3}$, and that the following quantities are uniformly bounded:
\begin{eqnarray*}
\left.\begin{aligned} 
\nonumber & |\hessq (\partial_t,\partial_i)+ tx_i r_q^{-3}|\text{, }
 |\hessq (\partial_t,\partial_s)+ ts r_q^{-3}|\\
\nonumber & |\hessq (\partial_s,\partial_i)+ sx_i r_q^{-3}|\text{, } 
|\hessq (\partial_s,\partial_s)- (r_q^{-1} - s^2 r_q^{-3})|\text{ and}\\
\nonumber & |\hessq (\partial_i,\partial_j)- (\delta_{ij} r_q^{-1} - x_i x_j r_q^{-3})|.
\end{aligned}\right.
\end{eqnarray*}
Then, we conclude that $||\hessq - (E_q^{-1})^{\ast}h||_g$ is uniformly bounded, where $(E_q^{-1})^{\ast}$ denotes a pull-back and $h$ is given by  
\begin{equation*}
h(v,w) = r_q^{-1}(\langle v, w \rangle - \langle v, \partial r_q \rangle \langle w, \partial r_q \rangle),
\end{equation*}
where $\partial r_q$ is the outward pointing unit normal vector to the Euclidean spheres in $T_qM$, and $\langle \cdot, \cdot \rangle$ represents the Euclidean inner product. Since $E_q^{-1}$ maps $\S_r(q)$ to Euclidean half-spheres, and the restriction of $h$ to these half-spheres coincides with their second fundamental  form with respect to the Euclidean metric, we conclude that $\B_r(p)$ are relatively convex.

The choice of the domain $Z$ mentioned in the statement of Lemma \ref{conical-replacements} is done in such a way that is has the properties listed in the following lemma.

\begin{lemma}\label{choice-Z}
For every $q\in \gamma$ and $0<\varepsilon <1$, there exists $Z$, neighborhood of $q$ in M, with the following properties. If $p\in \gamma\cap Z$, then the domain $Z$ can be parametrized using $E_p$, and we have
\begin{enumerate}
\item[(i)] $E_p : E_p^{-1}(Z) \rightarrow Z$ is a $C^2$ diffeomorphism

\item[(ii)] the set $Z$ is strictly relatively convex

\item[(iii)] $[\text{Lip}(E_p)]^n[\text{Lip}(E_p^{-1})]^n\leq 2$

\item[(iv)] $\text{Lip}(r_p |_{Z} )\leq 2 \ \ \ \ \ \ \ \ \ \ \ \ \ $

\item[(v)] $(E_p\circ \mulam \circ E_p^{-1})(y) \in Z$, whenever $y\in Z$ and $\lambda \in [0,1]$

\item[(vi)] $||\Lambda_n D (E_p \circ \mulam \circ E_{p}^{-1})(y)|| \leq \lambda^n (1+\varepsilon (1-\lambda))$, for all $y\in Z$, $\lambda\in [0,1]$.
\end{enumerate}
\end{lemma}

\begin{proof}
We choose $Z$ as a region of the form $\B_{r(q)}(q)$. From the previous arguments, items (i)-(iv) and (vi) follow immediately. On way to verify (v) is using the relative convexity of $Z$ to conclude that $E_p^{-1}(Z)$ is convex. 
\end{proof}

\begin{proof}[Proof of Lemma \ref{conical-replacements}]
First of all, we obtain inequalities similar to those presented in section 3.4(7) of \cite{pitts}. These facts provide us estimates for the masses of images of currents by the maps $(E_p\circ \mulam \circ E_p^{-1})$, and of the partial conical replacements that we use in this proof. Recall that $q, \varepsilon$, and $Z$ are fixed throughout the whole argument.

Let $p, r$, and $\delta$ be as in the statement, $\lambda\in [0,1]$, and $T^{\prime}$ be any current in $\Z_n(clos(\B_r(p)))$. Then, we have
\begin{equation}\label{supp-imageT-lambda}
spt((E_p\circ \mulam \circ E_p^{-1})_{\#}T^{\prime}) \subset clos(\B_{\lambda r}(p))
\end{equation}
and its mass can be estimated as
\begin{equation}\label{imageT-mass-lambda}
\M ((E_p\circ \mulam \circ E_p^{-1})_{\#}T^{\prime}) \leq \lambda^{n}(1+\varepsilon(1-\lambda))\M (T^{\prime}) \leq \M(T^{\prime}).
\end{equation} 
Moreover, if $T \in \Z_n(clos(\B_r(p)), \gamma \cup \S_r(p))$ satisfies the first two properties mentioned in item (c) of the statement, then, the partial conical replacement
\begin{equation}\label{def-partial-cone}
S_{\lambda} = (E_p)_{\#} \left( \delta_0 \mathbb{X} [(E_p^{-1})_{\#}\tilde{\partial}T - (\mulam\circ E_p^{-1})_{\#}\tilde{\partial}T]\right),
\end{equation}
where $\tilde{\partial}T = (\partial T)\llcorner \S_r(p) = \partial T - [|\gamma|]\llcorner \B_r(p)$, is such that
\begin{enumerate}
\item[(i)] $spt(\Slam) \subset \Atil (p, \lambda r, r) = \{\lambda r \leq r_p \leq r\}$

\item[(ii)] $\partial \Slam = \tilde{\partial} T - (E_p\circ \mulam \circ E_p^{-1})_{\#}(\tilde{\partial} T) + [|\gamma|]\llcorner \Atil (p,\lambda r, r)$

\item[(iii)] $\M (\Slam) \leq (\M(T)-\delta)(1-\lambda^n)$

\item[(iv)] 
\begin{eqnarray*}
 \M (\Slam + (E_p\circ \mulam \circ E_p^{-1})_{\#} T) \hspace{6cm}\\ 
 \hspace{4cm}\leq (\M(T)-\delta)(1-\lambda^n)+\lambda^{n}(1+\varepsilon(1-\lambda))\M (T) 
\end{eqnarray*}

\item[(v)] if, in addition, $T$ satisfies the third property in item (c) of the present lemma, then the inequality in (iv) reduces to
\begin{equation*}
\M (\Slam + (E_p\circ \mulam \circ E_p^{-1})_{\#} T) \leq \M(T).
\end{equation*}
\end{enumerate}
Compared to the proofs in 3.4(7) of \cite{pitts}, the only difference in the verification of (\ref{supp-imageT-lambda}), (\ref{imageT-mass-lambda}), and (i)-(v) above appears in the argument for part (ii). This happens because, in our case, we consider conical constructions over currents $\tilde{\partial} T$ which are not cycles. To overcome this minor difficulty, the only missing remark is as follows. Let us use $T_{\lambda} = (E_p^{-1})_{\#}\tilde{\partial}T - (\mulam\circ E_p^{-1})_{\#}\tilde{\partial}T$, then
\begin{eqnarray}\label{brdy-cone}
\partial \Slam = (E_p)_{\#} \left( \partial [\delta_0 \mathbb{X} T_{\lambda} ]\right) = (E_p)_{\#}(T_{\lambda}) - (E_p)_{\#}\left( \delta_0 \mathbb{X} [\partial T_{\lambda}]\right).
\end{eqnarray}
Also, observe that the definition of $\tilde{\partial}T$ implies that
\begin{eqnarray*}
\partial T_{\lambda} &=& (E_p^{-1})_{\#}\partial (\tilde{\partial}T) - (\mulam\circ E_p^{-1})_{\#}\partial (\tilde{\partial}T)\\
\nonumber & = & - (E_p^{-1})_{\#}\partial ([|\gamma|]\llcorner \B_r(p)) + (\mulam\circ E_p^{-1})_{\#}\partial ([|\gamma|]\llcorner \B_r(p)).
\end{eqnarray*}
Since $E_p$ restricted to $\gamma$ is simply a system of geodesic normal coordinates centered at $p$, we conclude that
\begin{equation*}
(E_p)_{\#}\left( \delta_0 \mathbb{X} [\partial T_{\lambda}] \right) = - [|\gamma|]\llcorner  \Atil (p, \lambda r, r).
\end{equation*}
This fact, combined with expression (\ref{brdy-cone}), allows us to conclude (ii).

The rest of the proof of the lemma follows the same arguments of Lemma 3.5 of \cite{pitts}, except for the fact that we need to deal with cones over currents that are not cycles. In our setting, the sequence $P_i$ of currents interpolating between $T$ and $(E_p)_{\#} ( \delta_0 \mathbb{X} (E_p^{-1})_{\#}((\partial T)\llcorner \S_r(p)))$ is defined by the exact same expressions as in Pitts' proof. The adaptations in this part boil down to considerations similar to those in the study of the boundary of $\Slam$. For instance, when proving that the current $P_i$ has boundary $\partial T$.
\end{proof}

\begin{proof}[Proof of Lemma \ref{choice-p_i}]
In this proof, we follow closely the ideas in Lemma 3.6 of \cite{pitts} and Lemma B.7 of \cite{Li-Zhou}. The rectifiability of $V$ in $Z$ implies that $(y, T_yV) \in G_n(Z)$ for $||V||$-almost all $y\in Z$, and that $y\mapsto T_yV$ is $||V||$-measurable. Which implies that the map
\begin{equation*}
F(p,y) = \langle \nabla r_p (y), T_yV\rangle\text{, } (p,y)\in (\gamma\cap Z)\times Z,
\end{equation*}
which is the orthogonal projection of $\nabla r_p(y)$ over the $n$-dimensional linear subspace $T_yV$, is $(\H^{n-1}\times ||V||)$-summable. Fubini's Theorem tells us that
\begin{equation*}
\int_{\gamma\cap Z} ||V||\{y\in Z\setminus (\partial M) : F(p,y)=0\} d\H^{n-1}(p)
\end{equation*}
is equal to
\begin{equation*}
\int_{Z\setminus (\partial M)} \H^{n-1}\{p \in \gamma \cap Z : F(p,y) =0\} d||V||(y).
\end{equation*}
Therefore, it suffices to show that $\H^{n-1}\{p \in \gamma \cap Z : F(p,y) =0\} = 0$, for $||V||$-almost all $y \in Z\setminus (\partial M)$. Indeed, this will imply that for $\H^{n-1}$-almost all $p \in \gamma\cap Z$ we have $||V||\{y\in Z \setminus (\partial M): F(p,y)=0\}=0$. In particular, we will be able to obtain, for those $p\in \gamma\cap Z$ and all $\rho>0$ with $\B_{\rho}(p) \subset\subset Z$,
\begin{eqnarray*}
||V||\S_{\rho}(p) & =  & ||V||\{y \in \S_{\rho}(p) : T_yV \subset T_y\S_{\rho}(p)\}\\
\nonumber         & = & ||V||\{y \in \S_{\rho}(p) : F(p,y) = 0\}\\
\nonumber         & \leq & ||V||\{y \in Z\setminus (\partial M) : F(p,y) = 0\} = 0.
\end{eqnarray*}

In the rest of this proof, we will make use of the map $\pi : Z \rightarrow \gamma \cap Z$ defined by $\pi(x,s,t) = (x,0,0)$, in any of our coordinates. In other words, it is the map that takes $y \in Z$, first consider its nearest point projection $(x,s,0)$ in $\partial M$, and, finally, returns the nearest point projection of $(x,s,0)$ in $\gamma$. In particular, $\pi(y)$ does not depend on the center of the coordinates.

From now on, fix $y = (x,s,t)\in Z\setminus (\partial M)$. It follows, by definition, that $r_p(y)^2 = t(y)^2+s(y)^2 + (r_p\circ \pi)(y)^2$, for every $p\in \gamma \cap Z$. Then, we can write
\begin{equation*}
r_p(y) \nabla r_p(y) = t(y) \nabla t(y) + s(y) \nabla s (y) + (r_p\circ \pi)(y) \nabla (r_p\circ \pi)(y).
\end{equation*}
Since the vanishing of $F(p,y)$ is equivalent to $\nabla r_p(y) \in (T_yV)^{\perp}$, we conclude that $\{p\in \gamma\cap Z : F(p,y)=0\}$ is a subset of
\begin{equation}\label{contained-affine}
\{p\in \gamma\cap Z: (r_p\circ \pi)(y) \nabla (r_p\circ \pi)(y) \in \xi(y) + (T_yV)^{\perp}\},
\end{equation}
where $\xi(y) = - t(y) \nabla t(y) - s(y) \nabla s (y)$, and $\xi(y) + (T_yV)^{\perp}$ represents the translation of $(T_yV)^{\perp}$ by $\xi(y)$. Observe also that
\begin{eqnarray*}
(r_p\circ \pi)(y) \nabla (r_p\circ \pi)(y) &=& (r_p\circ \pi)(y) (d\pi_y)^{\ast} \left( \nabla r_p (\pi(y))\right)\\
\nonumber &=& (d\pi_y)^{\ast} \left( (r_p\circ \pi)(y) \nabla^{\gamma} r_p (\pi(y))\right),
\end{eqnarray*}
where $(d\pi_y)^{\ast}$ is the adjoint, with respect to inner products induced by $g$, of the linear map $d\pi_y: T_yM \rightarrow T_{\pi(y)}\gamma$. One can easily see that $d\pi_y$ is onto for $||V||$-almost all $y\in Z$. Then, $(d\pi_y)^{\ast}$ is independent of the point $p$ and injective. The last property of the vector $\nabla (r_p\circ \pi)(y)$ that we need to remark is that it belongs to the $n$-dimensional subspace $[span(\partial/\partial t)]^{\perp}$, the orthogonal complement of the space spanned by $(\partial/\partial t)(y)$.

Observe that the $1$-dimensional affine space $\xi(y) + (T_yV)^{\perp}$ is also independent of $p$, and that it is not contained in $[span(\partial/\partial t)]^{\perp}$. Indeed, $\xi(y)$ is in the first space but does not belong to the second because of its $t(y)\nabla t(y)$ component, which does not vanish for $y \in Z\setminus (\partial M)$. Therefore, the set
\begin{equation*}
A_y = ((d\pi_y)^{\ast})^{-1}(\xi(y) + (T_yV)^{\perp})\cap [span(\partial/\partial t)]^{\perp}) \subset T_{\pi(y)}\gamma
\end{equation*}
is either empty, or contains a unique point, does not dependent on $p$, and has the property that
\begin{equation}\label{contained-affine-2}
\{p\in \gamma\cap Z : F(p,y)=0\} \subset \{p\in \gamma\cap Z: (r_p\circ \pi)(y) \nabla^{\gamma} r_p (\pi(y)) \in A_y\}.
\end{equation}
Since $r_p|_{\gamma}$ is the distance function to $p$ in $\gamma$, we have that 
\begin{equation*}
r_p(x) \nabla^{\gamma} r_p(x) = - (exp_{x}^{\gamma})^{-1}(p),
\end{equation*}
for every $x \in \gamma$, where $exp_x^{\gamma}$ denotes the usual exponential of $\gamma$ at $x$. This and the fact expressed in (\ref{contained-affine-2}) together allow us to conclude that
\begin{equation*}
\{p\in \gamma\cap Z : F(p,y)=0\} \subset exp_{\pi(y)}^{\gamma}(-A_y)\cap Z.
\end{equation*}
Since $n-1 = \text{dim}(\gamma)\geq 1$ and the set on the right hand side of the inclusion above has at most one point, we conclude that $\H^{n-1}(exp_{\pi(y)}^{\gamma}(-A_y)\cap Z) = 0$, for our choices of $y$. Then, the lemma is proved.
\end{proof}


\section{Pull-tight}\label{sect-pull-tight}

In this short section, we recall the stationarity condition introduced in \cite{dl-r}, and include the analog in our setting of the tightening map. We include also a result that relates convergence in different topologies of integral currents with boundary contained in $\gamma$. This is the content of the following lemma, which is the equivalent of Lemma 4.1. of \cite{MN-willmore}, and plays a role in the pull-tight deformation.

\begin{lemma}\label{lemma4.1ofMN}
Let $\mathcal{S}$ be a subset of $\Z_k(M,\gamma)$ of currents with uniformly bounded masses which is compact with respect to the $\F$-metric topology. For every $\varepsilon >0$, there exists $\delta >0$ with the following property: for $T \in \Z_k(M,\gamma)$ and $S\in \mathcal{S}$ with $\Flat_M(T-S) \leq \delta$ and $\M (T) \leq \M(S) + \delta$, then $\F(T,S)\leq \varepsilon$.
\end{lemma}

\begin{proof}
The proof of this lemma follows the same steps of that presented in \cite{MN-willmore}. We sketch that argument here to show that the presence of boundaries is not an issue. First of all, we observe that if $T$, $T_1$, $T_2, \ldots \in \Z_k(M,\gamma)$ have uniformly bounded masses, then $\F(T,T_i)\rightarrow 0$ if and only if $\Flat_M(T_i - T)\rightarrow 0$ and $\M(T) = \lim_i \M(T_i)$. Indeed, by definition, $\F(T,T_i)\rightarrow 0$ is equivalent to validity of the following two things $\Flat_M(T_i - T)\rightarrow 0$ and the induced varifolds $|T_i|$ converge to $|T|$ in the weak topology. Then, the continuity of the mass with respect to the varifold convergence implies the ``only if" part of the claim. To prove the ``if" part of it, suppose that $\Flat_M(T_i-T)\rightarrow 0$ and $\lim_i \M(T_i) = \M(T)$. Observe that whenever a subsequence $|T_j|$ converges to a varifold $V$, we have $V= |T|$. Indeed, one can check that the following properties hold:
\begin{itemize}
\item $\cup_j \ \ spt(T_j)\subset M$, which is bounded;

\item $\Flat_M(T_j - T)\rightarrow 0$, implies that $T_j$ converges to $T$ with respect to the flat topology in $I_k(\R^N)$, see Section 31 of \cite{simon}; and 

\item $||T_j||(\R^N) = \M(T_j)$ converges to $||V||(\R^N)$ because of the varifold convergence, and to $\M(T)$ by hypothesis.
\end{itemize} 
In particular, $||V||(\R^N)=\M(T)$, and the claim follows as an application of the result stated in 2.1(18)(f) of \cite{pitts}. Therefore, the usual compactness of the space of varifolds proves that the sequence $|T_i|$ converges to $|T|$.

Now we use the remark on the previous paragraph to sketch the proof of the lemma. First, one can prove that given $\varepsilon>0$ and $S\in \Z_k(M,\gamma)$, there exist $r,\delta>0$ such that: for $S^{\prime}\in B^{\F}_r(S)$ and $T\in \Z_k(M,\gamma)$ with $\Flat_M(T-S^{\prime}) \leq \delta$ and $\M (T) \leq \M(S^{\prime}) + \delta$, then $\F(T,S^{\prime})\leq \varepsilon$. This is proved by contradiction. Letting $S_i, T_i \in \Z_k(M,\gamma)$, and $\delta_i \searrow 0$, be such that $\M(T_i)\leq \M(S_i) + \delta_i$, $\Flat_M(S_i-T_i)\rightarrow 0$, $\F(S_i,S)\leq \delta_i$, and $\F(T_i, S_i)$ does not approach zero, we can prove that $T_i$ converge to $S$ in the flat topology of currents in $\R^N$, and $\F(|S|,|T_i|)$ does not approach zero. As another application of the result in 2.1(18)(f) of \cite{pitts}, we obtain that $\M(S)=\lim_i \M(T_i)$. This is a contradiction because of the claim proved in the previous paragraph. To conclude the proof of the lemma, we apply a finite covering argument. 
\end{proof}


\begin{defi}\label{stationary}
Let $U \subset M$ be a relatively open subset of $M$. We use $\V_s(U, \gamma)$ to denote the set of varifolds $V \in \V_n(M)$ for which $\delta V(X) \geq 0$, for all $X \in \mathfrak{X}(M)$ satisfying $spt(X) \subset U$, $X|_{\gamma} = 0$, and $g(X, \nu) \geq 0$ at $\partial M$ (recall that $\nu$ is the inward pointing unit vector normal to $\partial M$). 
\end{defi}

This stationarity condition is the same used by De Lellis and Ramic in \cite{dl-r}. Observe that $V \in \V_s(U, \gamma)$ implies that $\delta V (X) = 0$, for every non-exterior vector field supported away from $\gamma$; i.e., $X \in \mathfrak{X}(M)$ compactly supported in $U\setminus \gamma$, and with $g(X, \nu) \geq 0$ at $\partial M$. Indeed, it follows from the maximum principle that $spt(||V||)\cap (\partial M)\subset \gamma$. Let $X$ be a vector field supported away from $\gamma$. Consider a smooth function $f : M \rightarrow [0,1]$ such that $f = 1$ on a neighborhood of $spt(||V||)$, and $f = 0$ on a neighborhood of $spt(X)\cap (\partial M)$. Note that $\delta V(f X) = \delta V(X)$ and $f X = 0$ at $\partial M$. Since, $f X$ and $-f X$ are both admissible, we conclude that $\delta V(-f X) = \delta V(f X) = 0$.

In particular, any $V \in \V_s(U, \gamma)$ is stationary with respect to variations supported in $U\setminus (\partial M)$. The advantage of working with definition \ref{stationary} in our setting is that
\begin{equation*}
||V||(\gamma) = 0, \text{ for every } V \in \V_s(M, \gamma).
\end{equation*}
This property turns out to be very important for the application of the boundary regularity theory of Allard \cite{al-2}. Its proof is the content of Lemma 6.4 of \cite{dl-r}, and consist of the evaluation of the first variation of $V$ at a vector field tangential to $\partial M$, whose flow looks like a contraction of a small tubular neighborhood of $\gamma$ in $M$, along geodesics emanating from $\gamma$.

\begin{prop}\label{pull-tight}
Let $\Pi\in \Pi_m^{\#}(Z_n(M,\gamma), P_0)$ be a class of sweepouts with
\begin{equation*}
\LL (\Pi ) >  \sup \{\M(P_0(x)) : x \in \partial I^m\}, 
\end{equation*}
and $\{\Phi_i\}_i\in \Pi$ be a critical sequence. Then, there exists $\{\Psi_{i}\}_i \in \Pi$, which is also min-max, and such that 
\begin{equation*}
Crit(\{\Psi_i\}_i)\subset Crit(\{\Phi_i\}_i) \cap \V_s(M, \gamma).
\end{equation*} 
\end{prop}

This proposition is analogous to the classical pull-tight deformation of the Almgren-Pitts min-max theory. Its proof follows the same steps of that standard construction, such as in \cite{MN-willmore} or in \cite{C-DL}. For the reader's convenience, we include a sketch of the pull-tight construction adapted to our setting.

\begin{proof}[Sketch of the proof of Proposition \ref{pull-tight}]
Given $\{\Phi_i\}_i$, consider the following: $C = \sup \{ \M (\Phi_i(x))\}< \infty$, where the supremum is taken over all $i\in \N$ and $x$ in the domain of the maps $\Phi_i$. We construct a tightening map defined on
\begin{equation*}
A = \{V \in \V_n(M) : ||V||(M)\leq 2C\},
\end{equation*}
which is a compact set of varifolds, with fixed set the, also compact, region
\begin{equation*}
A_0 = \left( A \cap \V_s(M, \gamma)\right) \cup \{|P_0(x)| : x \in \partial I^m\}. 
\end{equation*}

Observe that for every $V \in A\setminus A_0$, there exists $X \in \X(M)$ such that $\delta V(X) <0$, $X|_{\gamma} = 0$ and $g(X,\nu) \geq 0$ at $\partial M$. Then, following the usual construction of the pull-tight deformation for our choice of fixed set, and using that the set of admissible vector fields is convex, we obtain smooth maps $\Omega_V : [0,1] \times M \rightarrow M$ satisfying the following properties:
\begin{enumerate}
\item[(i)] $\Omega_V(t,\cdot)$ maps $M$ diffeomorphically onto $\Omega_V(t,M)$, for all $V \in A$ and $t \in [0,1]$. Actually, it is the truncated flow of a non-exterior vector field that vanishes along $\gamma$;

\item[(ii)] $\Omega_V$ is continuous on $V \in A$; i.e., the smooth vector field described on (i) depends continuously on the varifold $V\in A$, the set $A$ being considered with the $\F$-metric, and the space $\X(M)$ with the topology of the $C^k$-seminorms (or simply with the $C^1$-topology);

\item[(iii)] $\Omega_V(t,x)=x$, whenever $t=0$ or $V \in A_0$;

\item[(iv)] $||\Omega_V(1,\cdot)_{\#} V||(M) < ||V||(M)$, for $V \in A\setminus A_0$;

\item[(v)] $\Omega_V(t,x)=x$,  for every $V \in A$, $x\in \gamma$, and $t\in [0,1]$.
\end{enumerate}
It follows from property (v) that, for all $V\in A$ and $t\in [0,1]$, we have
\begin{equation*}
\Omega_V(t,\cdot)_{\#}T \in \Z_n(M,\gamma),
\end{equation*}
whenever $T \in \Z_n(M,\gamma)$. Therefore, we obtain a map $H(t,T)$ defined, on pairs $t\in [0,1]$ and $T\in \Z_n(M,\gamma)\cap \{S : |S| \in A\}$, by the expression
\begin{equation*}
H(t,T) = \Omega_{|T|}(t,\cdot)_{\#}T \in \Z_n(M,\gamma).
\end{equation*}
This map has the same properties of its analogue in the proof in Section 15 of \cite{MN-willmore}, namely: it is continuous in the product topology, where $\Z_n(M,\gamma)$ is considered with the $\F$-metric topology, and satisfies:
\begin{itemize}
\item $\M(H(1,T))< \M(T)$, unless $|T| \in A_0$

\item $H(t,T) = T$, whenever $|T| \in A_0$, and $t \in [0,1]$

\item for every $\varepsilon>0$, there is a positive $\delta$ for which: if $x\in \partial I^m$, $t\in [0,1]$, and $\F(T, P_0(x))<\delta$, then $\F(H(t,T), P_0(x))<\varepsilon$. 
\end{itemize}

We apply this map $H$ to deform $\{\Phi_i\}_i$ into $\{\Psi_i\}_i$, the desired critical sequence, following the same steps in pages 766-768 of \cite{MN-willmore}. As observed by the authors, in that argument, they need to interpolate to obtain a true competitor. Our adaptation is possible because where Theorems 14.2 and 13.1, and Lemmas 4.1 and 7.8 of \cite{MN-willmore}, which are suited to cycles, are used in that proof, we can apply our Theorems \ref{interpolation_boundary} and \ref{discretization}, Lemma \ref{lemma4.1ofMN}, and the fact expressed in (\ref{brdy-values-sweepouts}) of section \ref{minmax-sect}, respectively. We note that our application of Theorem \ref{interpolation_boundary} is possible because $\partial \Phi_i(x) = \gamma$, for all $x\in dmn(\Phi_i)$ and large $i\in \N$. This was also observed in section \ref{minmax-sect}.

\end{proof}

\section{The almost minimizing condition}\label{sect-amc}

In this section, we introduce the notion of almost minimizing varifolds which is adequate to the present setting. As in the classical boundaryless scenario, it plays a crucial role in the regularity theory. After we describe the basic notions and verify that it shares very similar properties with that of \cite{pitts}, we present the parts of the Almgren-Pitts min-max program that involve the almost minimizing condition in some way. In subsection \ref{exist-reg-replacements}, we describe the adapted construction of comparison surfaces, and prove their regularity properties. In subsection \ref{comb-argument}, we briefly discuss the combinatorial argument. We begin with some notation.

\begin{defi}
Let $U\subset M$ be relatively open, $\varepsilon, \delta$ be positive numbers, and $\nu$ be either $\Flat_C$, $\F$, or $\M$. We use $\A(U;\varepsilon, \delta; \nu)$ to denote the set of all $T \in \Z_n(M, (M\setminus U)\cup \gamma)$ with $(\partial T)\llcorner U = [| \gamma |]\llcorner U$ and for which the following holds: if $T = T_1, T_2, \ldots, T_q \in \Z_n(M, (M\setminus U)\cup \gamma)$ satisfies
\begin{itemize}
\item $(\partial T_i)\llcorner U = [| \gamma |]\llcorner U$
\item $spt (T_i - T)\subset U$
\item $\nu (T_i, T_{i-1})\leq \delta$
\item $\M(T_i)\leq \M(T)+\delta$,
\end{itemize}
then, all $T_i$ have mass at least $\M(T)-\varepsilon$. In particular, $\M(T_q)\geq \M(T)-\varepsilon$.
\end{defi}

Next, we explain the almost minimizing condition that we apply.

\begin{defi}
Let $V\in \V_n(M)$ be a varifold in $M$. We say that $V$ is almost minimizing in $U$ if for every $\varepsilon >0$, there exist $\delta >0$ and $T \in \A(U;\varepsilon, \delta; \Flat_C)$ such that $\F_{U}(|T|, V)< \varepsilon$. 
\end{defi}

These notions are the analogues in our setting to the notions introduced in Section 3.1 of the pioneering work of Pitts \cite{pitts}. The first property that we observe is the stability of almost minimizing varifolds. More precisely, we have the following lemma.

\begin{lemma}\label{am-implies-stationary}
If $V\in \V_n(M)$ is almost minimizing in $U$, then $V\in \V_s(U, \gamma)$ in the sense of Definition \ref{stationary}. Moreover, $V$ is stable with respect to non-exterior variations in $U$ that are supported away from $\gamma$.
\end{lemma}

The first part of the proof of Theorem 3.3 of \cite{pitts}, the stationarity of $V$, can adapted in a straightforward manner in the verification of Lemma \ref{am-implies-stationary}. Indeed, it is clear that the $1$-parameter family of diffeomorphisms induced by such vector fields preserve the boundary condition.

As discussed in \cite{pitts}, see pages 97 and 98, for the existence theory it is more convenient to consider sequences that are fine in $\M$. But for the construction of replacements, which is discussed later in this section, and for all the regularity theory, the norm induced by $\Flat_C$ is more suitable. Since the notion of almost minimizing varifolds is used in both part of the argument, we must check that it is essentially the same of that that one would obtain using approximating currents in $\A(U;\varepsilon, \delta; \F)$ or $\A(U;\varepsilon, \delta; \M)$, instead of those that we used. This fact is the content of the following theorem.

\begin{thm}\label{am-equiv-defs}
Let $U\subset M$ be relatively open, and $V\in \V_n(M)$. Each of these items implies the one that follows it.
\begin{enumerate}
\item[(I)] $V$ is almost minimizing in $U$

\item[(II)] for every $\varepsilon >0$, there exists a pair $(\delta, T)$, where $\delta$ is a positive number and $T\in \A(U;\varepsilon, \delta; \F)$ with $\F_{U}(|T|, V)< \varepsilon$

\item[(III)] for every $\varepsilon >0$, there exists a pair $(\delta, T)$, where $\delta$ is a positive number and $T\in \A(U;\varepsilon, \delta; \M)$ with $\F_{U}(|T|, V)< \varepsilon$

\item[(IV)] $V$ is almost minimizing in any $W\subset\subset U$, relatively open in $M$.
\end{enumerate} 
\end{thm}

The proof of this result follows the exact same lines of that of Theorem 3.9 in \cite{pitts}. It is a direct application of our Lemma \ref{interp.flat-to-mass}, instead of Pitts' Lemma 3.8. We omit the details here.

\subsection{Existence and regularity of replacements}\label{exist-reg-replacements} In this part, we adapt an important ingredient of the regularity theory; we construct comparison hypersurfaces to almost minimizing varifolds.

Let $V$ be almost minimizing in $U$ and $K\subset U$ be a compact subset. Consider sequences $\{\varepsilon_i\}$ and $\{\delta_i\}$ of positive numbers decreasing to $0$, and $T_i \in \A(U;\varepsilon_i, \delta_i; \Flat_C)$, with $\F_{U}(|T_i|, V) \leq \varepsilon_i$.

Fix $i \in \N$, and consider all finite sequences $\{T^j_i\} \subset \Z_n(M, (M\setminus U)\cup \gamma)$, $1\leq j \leq q$ ($q$ is any positive integer, and it is not fixed), satisfying
\begin{itemize}
\item $(\partial T^j_i)\llcorner U = [| \gamma |]\llcorner U$
\item $spt (T^j_i - T_i)\subset K$
\item $\Flat_C (T^j_i - T^{j-1}_i)\leq \delta_i$
\item $\M(T^j_i)\leq \M(T_i)+\delta_i$.
\end{itemize}
There is one such sequence for which the final mass $\M(T^q_i)$ is minimal. The existence of such a minimizer is a consequence of the compactness theorem for integer rectifiable currents. The existence of a boundary in $U$ is not an issue at this point because, in $U$, all such currents have the same boundary $\gamma$. Let us use $S_i$ to denote the final current of such an optimal sequence, $S_i = T^q_i$. Next, we list some properties of the currents $S_i$.
\begin{enumerate}
\item[(a)] $\M (T_i) - \varepsilon_i \leq \M (S_i) \leq \M (T_i)$

\item[(b)] $S_i \in \A(U;\varepsilon_i, \delta_i; \Flat_C)$

\item[(c)] $S_i \llcorner (\R^N \setminus K) = T_i \llcorner (\R^N \setminus K)$

\item[(d)] $\M(S_i) \leq \M(S)$, for all $S\in \Z_n(M, (M\setminus U)\cup \gamma)$ such that
\begin{equation*}
(\partial S)\llcorner U = [| \gamma |]\llcorner U, \ \ spt (S - T_i)\subset K,\text{ and }\Flat_C (S - S_i)\leq \delta_i
\end{equation*}

\item[(e)] $|S_i|$ is stable in $int(K)\setminus (\partial M)$; i.e., the second variation of the mass $\delta^2|S_i|$ is non-negative with respect to all vector fields $X$ supported in $int(K)\setminus (\partial M)$. Moreover, $|S_i|\in \V_s(int(K), \gamma)$, see definition \ref{stationary}.

\item[(f)] for every $p \in int(K)$, including points in $(\partial M)\cap int(K)$, there exists $r>0$ for which: $\M(S_i) \leq \M(S)$, for all $S\in \Z_n(M, (M\setminus U)\cup \gamma)$ such that $(\partial S)\llcorner U = [| \gamma |]\llcorner U$, and $spt (S - S_i)\subset clos(B^N_r(p))$, where $B^N_r(p)$ denotes the open Euclidean ball centered at $p$ of radius $r$;

\item[(g)] using the notation of the previous item, we conclude that: 
\begin{equation*}
\M (S_i) \leq \M (T+S_i), \text{ for all } T \in \Z_n(M) \text{ with } spt (T)\subset clos(B^N_r(p)).
\end{equation*}

\item[(h)] $(S_i) \llcorner G_n(int(K))$ is, up to multiplicity, the integral current corresponding to $\Sigma_i$, where the latter is an embedded smooth minimal hypersurface except for a closed set $sing(\Sigma_i)$ of Hausdorff dimension at most $n-7$, and $sing(\Sigma_i) = \overline{\Sigma_i} \setminus \Sigma_i$. Moreover, $sing(\Sigma_i)\cap (\partial M) = \varnothing$ and $(\partial \Sigma_i)\cap int(K) = \gamma\cap int(K)$. In particular, it follows that the connected components of $\Sigma_i$ that intersect $\partial M$ have multiplicity one. 
\end{enumerate}

Observe that item (g) above says that $S_i$ is locally area-minimizing with boundary $\gamma$ in $int(K)$. Then, we note on (h) the regularity of such currents. More precisely, the interior regularity of $S_i$ at interior points follows from the classical theory of codimension one area-minimizing currents. The boundary regularity was established by Allard, see the corollary announced in \cite{al-1} and the theory developed in \cite{al-2}. Allard's results can be applied at this point because the boundary of the space $M$ is strictly convex. This application was also observed in the proof of Corollary 9.7 of \cite{dl-r}. See also \cite{DS}.

The property stated in item (e) was not applied in the arguments of the subsequent facts. The verification of the validity of (e) uses (d). First of all, we observe that the varifold $|S_i|$ is stationary in the open subset $int(K)\setminus (\partial M)$. If this were not true, we would have a vector field $X \in \mathfrak{X}(M)$ supported in $int(K)\setminus (\partial M)$, for which $\delta |S_i|(X)<0$. The variation of $S_i$ in the direction of $X$ would produce a competitor in the sense of (d), since it would preserve the boundary constraint, of least area. Then, we conclude the stationarity part. Once we know this, the argument by contradiction to prove stability is analogous; using $\delta |S_i|(X)=0$ and $\delta^2 |S_i|(X,X)<0$. Moreover, this argument can also be used to prove the second claim of (e), namely $|S_i|\in \V_s(int(K), \gamma)$.

Let $V_i$ be the varifold obtained by
\begin{equation}\label{pre-replacements}
V_i = |S_i|\llcorner G_n(U) + V\llcorner \big( G_n(\R^N)\setminus G_n(U)\big).
\end{equation}
Suppose, up to restriction to a subsequence, that the $V_i$ converge as varifolds to $\tilde{V}$. Any such limit is called a replacement for $V$ in $K$.

Observe that we were able to perform the above construction in $K$ using that $V$ is almost minimizing in some relatively open subset $U$ containing $K$ only. Thus, we conclude the following:

\begin{prop}\label{exist-replacements}
Let $U\subset M$ be relatively open, and $K\subset U$ be compact. If $V\in \V_n(M)$ is almost minimizing in $U$, then it has a replacement in $K$.
\end{prop}

Next, we list some properties of these replacements. Recall the notation $\V_s(M, \gamma)$ introduced in Section \ref{sect-pull-tight} to denote the set of varifolds that are stationary with respect to non-exterior variations vanishing along $\gamma$.

\begin{thm}\label{replacement-properties}
Let $U\subset M$ be relatively open, $V \in \V_n(M)$ be almost minimizing in $U$, $K\subset U$ be compact, and $\tilde{V}$ be a replacement of $V$ in $K$. Then:
\begin{enumerate}
\item[(i)] $V\llcorner G_n(\R^N\setminus K) = \tilde{V} \llcorner G_n(\R^N\setminus K)$;

\item[(ii)] $\tilde{V}$ is almost minimizing in $U$;

\item[(iii)] $||V||(M) = ||\tilde{V}||(M)$;

\item[(iv)] $\tilde{V}\llcorner G_n(int(K))$ is, up to multiplicity, a stable minimal hypersurface $\Gamma$ except for a closed set $sing(\Gamma)$ of Hausdorff dimension at most $n-7$, with $(\partial M) \cap sing(\Gamma) = \varnothing$, and $(\partial \Gamma)\cap int(K) = \gamma \cap int(K)$. 

In particular, the connected components of $\Gamma$ which intersect $\partial M$ have multiplicity one;

\item[(v)] if $V\in \V_s(M, \gamma)$, then $\tilde{V} \in \V_s(M, \gamma)$.
\end{enumerate}
\end{thm}

The proof of this theorem involves applications of some of the main results obtained by De Lellis and Ramic in \cite{dl-r}. In order to show how their ideas can be used for the objects that we have in hand, we recall some notions and fact introduced in that work. Let us start with an important definition. 

\begin{defi}\label{opening-dr}
Let $K_1\subset M$ be closed and $\theta \in (0, \pi/2)$. We say that $K_1$ meets $\gamma$ at an opening angle at most $\theta$ if: $K_1\cap (\partial M) = K_1\cap \gamma$, and for every $C^1$ curve $\alpha : [0,1]\rightarrow K_1$, with $p = \alpha (0) \in \gamma\cap K_1$, we have
\begin{equation*}
|(T_{p}\partial M)(\alpha^{\prime}(0))|\leq |(T_{p}\partial M)^{\perp} (\alpha^{\prime}(0))|\cdot \tan \theta,
\end{equation*}
where $T_{p}\partial M$ and $(T_{p}\partial M)^{\perp}$ represent the orthogonal projections of $T_pM$ over each one of the corresponding subspaces. 
\end{defi}

The next result is Lemma 8.1 of \cite{dl-r}, it shows that every varifold which is stationary with respect to non-exterior variations that vanish along $\gamma$, see our Definition \ref{stationary}, is contained in a wedge-like region along the submanifold $\gamma$, in the sense of the above definition.

\begin{lemma}\label{lemma8.1-dr}
Let $M^{n+1}$, $\gamma^{n-1}\subset \partial M$ be spaces as before, and $U_2\subset\subset  U_1$ two relatively open subsets of $M$, or $U_1=U_2=M$. Then, there exist a constant $\theta \in (0, \pi/2)$ and a compact set $K_1 \subset \overline{U_2}$ satisfying the following statements:
\begin{enumerate}
\item[(a)] $K_1$ meets $\gamma$ at an opening angle at most $\theta$

\item[(b)] $spt(||V||)\cap U_2 \subset K_1$, for every $V\in \V_s(U_1, \gamma)$.
\end{enumerate} 
\end{lemma}

Finally, the last result from \cite{dl-r} that we need to recall, Theorem 7.4 of that paper, is the analog in our setting of the compactness of the space of stable minimal hypersurfaces. As in the works of Schoen, Simon, and Yau \cite{SSY}, Schoen and Simon \cite{SS}, and, recently adapted to the free-boundary setting, of Guang, Li, and Zhou \cite{GLZ}, the following theorem was obtained in \cite{dl-r} as a consequence of curvature estimates for stable minimal hypersurfaces.

\begin{thm}\label{thm7.4-dr}
Let $M^{n+1}$, $\gamma^{n-1}\subset \partial M$ be as before, $U_2\subset M$ relatively open, $\theta \in (0, \pi/2)$, and $K_1 \subset \overline{U_2}$ a compact set which meets $\gamma$ at an opening angle at most $\theta$. Let $\{\Gamma_i\}$ be a sequence of stable minimal hypersurfaces in $U_2$ which are smooth away from a closed set $sing(\Gamma_i)$ with:
\begin{eqnarray*}
\nonumber & \H^{n-2}(sing(\Gamma_i))=0, \gamma\cap sing(\Gamma_i) = \varnothing, (\partial \Gamma_i)\cap U_2 = \gamma \cap U_2,\\
\nonumber & \Gamma_i \subset K_1, \text{ and } \sup_i \H^{n}(\Gamma_i) < \infty.
\end{eqnarray*}
Then, up to subsequences, $\Gamma_i$ converges to a varifold $V^{\prime}$, which is an integer varifold and, up to multiplicity, it is a stable minimal hypersurface $\Gamma$ with
\begin{equation*}
dim(sing(\Gamma))\leq n-7, (\partial M) \cap sing(\Gamma) = \varnothing, \text{ and }  (\partial \Gamma)\cap U_2 = \gamma \cap U_2.
\end{equation*}
In particular, the connected components of $\Gamma$ which intersect $\partial M$ have multiplicity one. Moreover, the convergence is smooth away from $sing(\Gamma)$.
\end{thm}

The stability condition assumed in the above statement is with respect to the second variation of the mass, and for variations that are compactly supported in $U_2\setminus (\partial M)$. The convergence of $\Gamma_i$ to $V^{\prime}$ is as varifolds in $U_2$.

Now we present the proof of Theorem \ref{replacement-properties}.

\begin{proof}[Proof of Theorem \ref{replacement-properties}]
The facts stated in items (i)-(iii) easily follow from properties (a)-(c) of the currents $S_i$. For part (iv), the argument is as follows. For every $U_2\subset M$ relatively open with $U_2\subset \subset int(K)$, consider the constant $\theta = \theta(U_2, int(K))$ and the compact $K_1\subset \overline{U_2}$ which meets $\gamma$ at an opening angle at most $\theta$, as given by Lemma \ref{lemma8.1-dr}. Observe that the $V_i$ considered in (\ref{pre-replacements}) coincide with $|S_i|$ in $int(K)$. Then, it follows from the second claim of the property (e) of $S_i$ that $V_i \in \V_s(int(K), \gamma)$, and from Lemma \ref{lemma8.1-dr} that $spt(||V_i||)\cap U_2 \subset K_1$. Observe that property (h) of $S_i$ implies that $spt(||V_i||)\cap U_2$ is composed of stable minimal hypersurfaces $\Gamma_i = \Sigma_i\cap U_2$, which satisfy all the required hypothesis of Theorem \ref{thm7.4-dr}. Therefore, up to subsequences, $\Gamma_i$ converges to a varifold $V^{\prime}$ satisfying, in the domain $U_2$, all the properties that we want to obtain for $\tilde{V}$ in $int(K)$.

The fact that $\tilde{V}$ has those same properties in $U_2$ follows from that, because the only possible difference between $|S_i|\llcorner G_n(U_2)$ and $\Gamma_i$ is multiplicity, and $|S_i|\llcorner G_n(U_2)$ converges to a varifold which coincides with $\tilde{V}$ in $G_n(U_2)$, while $\Gamma_i$ converges to $V^{\prime}$. Since $U_2\subset \subset int(K)$ is the only restriction on $U_2$, we obtain all the desired properties for $\tilde{V}$ in $int(K)$, and part (iv) is proved.

For part (v), we observe that the restriction $\tilde{V}\llcorner G_n(\R^N\setminus K)$ and $\tilde{V}\llcorner G_n(U)$ are stationary with respect to admissible variations. This follows from (i) and $V\in \V_s(M, \gamma)$ for the restriction to $G_n(\R^N\setminus K)$, and from (ii) and Lemma \ref{am-implies-stationary} for the restriction to $G_n(U)$. Then, for every $X\in \mathfrak{X}(M)$ with $X|_{\gamma} = 0$, and $g(X, \nu) \geq 0$ at $\partial M$, we can simply use a partition of unity to decompose the vector field as a sum $X=X_1+X_2$ of admissible variations such that $spt(X_1)\subset M\setminus K$ and $spt(X_2)\subset U$. Finally, since $\delta \tilde{V}(X_1)=\delta \tilde{V}(X_2)=0$, we conclude that $\delta \tilde{V}(X)=\delta \tilde{V}(X_1) + \delta \tilde{V}(X_2)=0$, for all such $X$. This proves the lemma.
\end{proof}

\subsection{Main existence theorem}\label{comb-argument}

In this part, we explain the fundamental existence theorem. It guarantees the existence of limit varifolds that are simultaneously stationary, in the sense described in definition \ref{stationary}, and almost minimizing is small annuli centered at arbitrary points of $M$.

In what follows, we use $A(p,s,r)$ to denote some open annuli centered at points $p \in M$, of inner and outer radii $s$ and $r$, respectively. More precisely, we consider $M$ isometrically embedded in $\R^N$. We use $A(p,s,r)$ to denote the portion in $M$ of Euclidean annuli centered at $p$.

\begin{thm}\label{comb-arg}
Let $M$, $\gamma$, and $\Pi$ be as in the statement of Theorem\ref{teorema-B}. Then, there exists an integral varifold $V \in \V_n(M)$ such that:
\begin{enumerate}
\item[(a)] $||V||(M) = \LL (\Pi)$

\item[(b)] $V\in \V_s(M,\gamma)$

\item[(c)] for every $p\in M$, there exists a positive number $r_p$ such that $V$ is almost minimizing in $A(p,s,r)$, for all $0<s<r<r_p$.  
\end{enumerate}
\end{thm}

This result is the analog in the present setting of Theorem 4.10 of \cite{pitts}. The existence of a varifold with the desired properties is obtained after an argument by contradiction, which contains a complicated combinatorial part. The proof starts with the choice of a critical sequence $\{\varphi_i\} \in \Pi$ for which all min-max limits are in $\V_s(M, \gamma)$. In our setting, this was achieved in Proposition \ref{pull-tight}. In the contradiction argument, we assume that every $\varphi_i(x)$, with large $i\in \N$, and mass close to the min-max invariant,
\begin{equation*}
\M(\varphi_i(x)) \geq \LL (\Pi) - \delta,
\end{equation*}
admits finite deformations $T_0= \varphi_i(x), T_1, T_2, \ldots, T_q \in \Z_n(M, \gamma)$ which are are arbitrarily fine with respect to $\Flat_M$, supported in small annuli, preserving the boundary constraint, such that the masses $\M(T_j)$ are uniformly bounded by $\M(\varphi_i(x))$ plus an arbitrarily small value, and $\M(T_q) < \M(\varphi_i(x))-\varepsilon$, where $\varepsilon$ is a positive constant independent of $\varphi_i(x)$. It follows from Theorem \ref{am-equiv-defs} that we can assume that these finite sequences are fine in $\M$. In order to achieve a contradiction, we need to combine these mass-decreasing, well controlled, finite variations of big slices to obtain a better competitor.

The analog of Theorem \ref{comb-arg} in the work of De Lellis and Ramic \cite{dl-r} is explained in details in sections 4 and 5 of that paper. There the authors use a delicate argument, Lemma 5.1 of \cite{dl-r}, to explain how one can use a $1$-parameter deformation of a single slice with properties similar to those of our sequence $T_j$, to deform nearby slices. The key idea of that argument uses a tool called ``freezing", which was introduced in \cite{dl-t}.

In the discrete setting, the technique that is used in the corresponding step is a cut and paste argument, which uses $\Flat_M$-isoperimetric choices and slicing theory for integral currents. In Pitts' proof of Theorem 4.10 of \cite{pitts}, this is the content of part 9. As we have mentioned in previous sections, in light of the fact that $\partial \varphi_i(x) = \gamma$ for large $i\in \N$, we can perform that cut and paste argument using the same notion of $\Flat_M$-isoperimetric choices and the same formula to construct the deformations of nearby slices. See the formula that defines $T(j,2)$ on page 168 of \cite{pitts}. Since all $\Z_n(M,\gamma)$ currents considered in our deformation are defined by the same expression as in the case of $n$-cycles, we obtain the exact same properties.

The rest of the argument is lengthy, but it is purely combinatorial in the sense that no other operations on currents or deformations are needed. Therefore, we can combine the above sequences $T_j=T_j(\varphi_i(x))$ to deform $\{\varphi_i\}$ homotopically to a new sweepout $\{\psi_i\}\in \Pi$ with
\begin{equation*}
L(\{\psi_i\}) < L(\{\varphi_i\}) = \LL (\Pi).
\end{equation*}

\section{Tools for regularity and Proof of Theorem \ref{teorema-B}}\label{sect-teoremaB}

In this section we prove the min-max theorem stated in the introduction. More precisely, we give reasons for the fact that the varifold $V \in \V_n(M)$ obtained on Theorem \ref{comb-arg} satisfies the required properties. The conclusions about the support of $V$ near points $p\in M\setminus (\partial M)$ follow from the works of Pitts \cite{pitts} and Schoen and Simon \cite{SS}; i.e., near any such points $p$, the support of $V$ is an embedded minimal hypersurface that it is regular except for a set of Hausdorff dimension at most $n-7$. In order to achieve similar smoothness properties near the boundary, we explore the variational properties that we proved for $V$, and apply the theories developed in \cite{al-2} and \cite{dl-r}. The arguments presented in this section are very similar to those in the regularity part of \cite{dl-r}. We decided to include these details here because in their proof of Lemma 10.1 of that paper, the authors focused on the unconstrained case.


In the first part of the proof, we consider a varifold $V\in V_n(M)$ which is integral and is stationary in the sense of definition \ref{stationary}. We show that any varifold tangent to $V$ at a point $p\in \gamma\cap spt(||V||)$ is the sum, with positive integer coefficients, of finitely many $n$-dimensional linear half-spaces containing $T_p\gamma$. This is the content of Proposition \ref{k-half-spaces}. Its proof involves the combination of results from the works of Allard \cite{al-2} and De Lellis and Ramic \cite{dl-r}. We start by recalling some of those facts.

\begin{lemma}\label{monotonicity}
Let $V\in \V_n(M)$ be an integral varifold in $\V_s(M, \gamma)$. For every $p\in \gamma\cap spt(||V||)$, there exist a positive constant $\rho_0$ and a smooth function $\Psi(\rho)$ satisfying $\Psi(\rho)\rightarrow 0$, as $\rho\rightarrow 0$, and such that
\begin{equation*}
\rho \in (0, \rho_0) \mapsto e^{\Psi(\rho)}\frac{||V||B_{\rho}(p)}{\rho^n}
\end{equation*}
is monotone non-decreasing. In particular, the density of $||V||$ at $p$,
\begin{equation*}
\Theta^n(||V||,p) = \lim_{\rho\rightarrow 0} \frac{||V||B_{\rho}(p)}{\omega_n\rho^n}
\end{equation*} 
is well-defined. Here $\omega_n$ is the $n$-dimensional volume of an Euclidean $n$-ball of radius one. Moreover, $\Theta^n(||V||,p)$ is finite for all $p\in \gamma\cap spt(||V||)$.
\end{lemma}

As usual, we assume that $M$ is isometrically embedded in $\R^N$. For $p\in \R^N$ and $r>0$, we use the maps $\eta_{p,r}: \R^N \rightarrow \R^N$ defined by $\eta_{p,r}(y)= r^{-1}(y-p)$. It follows from Lemma \ref{monotonicity} that, for $p\in \gamma\cap spt(||V||)$, the masses of the varifolds $(\eta_{p,r})_{\#}V$ in the unit ball centered at the origin of $\R^N$ are uniformly bounded. In conclusion, there exists a sequence $r_k$ of positive numbers decreasing to zero such that $(\eta_{p,r_k})_{\#}V$ converges to a varifold $C$. Any such limit is called a tangent varifold of $V$ at $p$, and the space of all such objects is usually denoted by $\text{Var Tan}(V,p)$.

In the following lemma, we list some properties of $C \in \text{Var Tan}(V,p)$. We use $T^+_pM$ to denote the smooth limit of $\eta_{p,r_k}(M)$, as $k\rightarrow \infty$. Similarly, let $T_pM$ be the only $(n+1)$-dimensional linear subspace of $\R^N$ that contains the half-space $T^+_pM$, and $T_p\gamma$ be the $(n-1)$-subspace that represents the limit of the analogous images of the submanifold $\gamma$.

\begin{lemma}\label{var-tan}
Let $V\in \V_n(M)$ be an integral varifold in $\V_s(M, \gamma)$. For every $p\in \gamma\cap spt(||V||)$ and $C \in \text{Var Tan}(V,p)$, we have the following properties:
\begin{enumerate}
\item[(i)] $C$ is an integral varifold supported in $T^+_pM$
\item[(ii)] $C$ is stationary in $(T_pM)\setminus (T_p\gamma)$
\item[(iii)] $C$ is a cone; i.e., $(\eta_{0,r})_{\#}C = C$, for all $0<r<\infty$.
\end{enumerate}
\end{lemma}

Let $V$, $p$, and $C$ be as in the statement of Lemma \ref{var-tan}, and $W$ be the $2$-plane given as the orthogonal complement of $T_p\gamma$ in $T_pM$. We can apply Lemma 5.1 of \cite{al-2} for our choices of $C$ and $W$ in $\R^{n+1} = T_pM$. Let $S^n$ and $B_1^{n+1}(0)$ denote the unit sphere and ball in $T_pM$ with respect to the Euclidean inner product. For every $\varphi \in C^{\infty}(W\cap S^n)$, define $T(\varphi)$ by
\begin{equation*}
\int_{[B_1^{n+1}(0)\setminus (T_p\gamma)]\times G(n+1,n)} \varphi\bigg(\frac{W(x)}{|W(x)|}\bigg)\{w\in W: \langle w, x\rangle=0\}\cdot S \text{ } dC(x,S).
\end{equation*}
In the above expression, $G(n+1, n)$ is the space of $n$-dimensional linear subspaces of $T_pM$. We also follow the notation in subsection 2.3 of \cite{al-3}, where $W$, $\{w\in W: \langle w, x\rangle=0\}$, and $S$ also denote the orthogonal projections onto the respective subspaces, and the inner product $f \cdot g$ is the one defined as the trace of $f^{\ast}\circ g$ in the space $\text{Hom}(T_pM, T_pM)$ of homomorphisms of $T_pM$.

The above $T(\varphi)$ is the same defined in Allard's paper. It follows from the first part of his result that $T$ is a multiple of $\H^1\llcorner (W\cap S^n)$. On the other hand, $C$ is supported in $T^+_pM$, which is a half-space. Therefore, we conclude that the multiple must be zero, and part (2) of Lemma 5.1 in \cite{al-3} implies that $W(spt(||C||))\cap S^n$ is finite. In particular, $spt(||C||)$ is contained in a union of finitely many $n$-dimensional half-spaces $\pi_i$ that meet at $T_p\gamma$.

Finally, we observe that the fact that $C$ is stationary in $(T_pM)\setminus (T_p\gamma)$, together with the constancy theorem, Theorem 41.1 of \cite{simon}, imply that the multiplicity of $C$ is constant over each $\pi_i$. Summarizing, we have:

\begin{prop}\label{k-half-spaces}
Let $V\in \V_n(M)$ be an integral varifold that belongs to the class $\V_s(M, \gamma)$, introduced in definition \ref{stationary}, and $p\in \gamma\cap spt(||V||)$. If $C \in \text{Var Tan}(V,p)$, then there exist collections $\{\pi_i\}_{i=1}^k$ of $n$-dimensional half-spaces that contain $T_p\gamma$, and $\{c_i\}_{i=1}^k\subset \N$ such that $C = \sum_{i=1}^k c_i \pi_i$.
\end{prop}


From now on, we focus our attention on the varifold $V\in \V_n(M)$ obtained on Theorem \ref{comb-arg}. It follows from previous discussions that, away from the boundary of $M$, $V$ is smooth outside a small set. On the other hand, recall that the stationarity property $V \in \V_s(M, \gamma)$ implies that
\begin{equation*}
spt(||V||)\cap (\partial M) \subset \gamma \text{ and } ||V||(\gamma)=0.
\end{equation*}
Therefore, we conclude that $V$ is an integral varifold in $M$.


In the following proposition, we prove that the replacements of $V$ in small annuli also satisfy properties similar to those of $V$ stated in Theorem \ref{comb-arg}. We use the numbers $r_p$ that were introduced on part (c) of that result.

\begin{prop}\label{replacement-further-prop}
Let $M, \gamma, \Pi$, and $V \in \V_n(M)$ be as in Theorem \ref{comb-arg}. Fix any annulus $An = A(p,s,r)$, with $p\in M$ and $0<s<r<r_p$. Then, there exists a replacement $\tilde{V}$ for $V$ in $clos(An)$. Moreover, $\tilde{V} \in \V_s(M, \gamma)$, $||\tilde{V}||(M) = \LL (\Pi)$, and, for every $q \in M$, there exists a positive number $r^{\prime}_q$ such that $\tilde{V}$ is almost minimizing in $A(q,s,r)$, for all $0<s<r<r^{\prime}_q$. 
\end{prop}

\begin{proof}
Let us use $K$ to denote the compact subset $clos(An)$. Choose $U = A(p,s^{\prime}, r^{\prime})$ such that these positive radii satisfy $s^{\prime}<s$ and $r<r^{\prime}<r_p$. Observe that $K\subset U$ and $V$ is almost minimizing in $U$. It follows from Proposition \ref{exist-replacements} and Theorem \ref{replacement-properties} that $V$ has a replacement $\tilde{V}$ in $K$, which satisfies the following properties:
\begin{itemize}
\item $\tilde{V} \in \V_s(M, \gamma)$
\item $||\tilde{V}||(M) = ||V||(M) = \LL (\Pi)$
\item $\tilde{V}$ is almost minimizing in $U$.
\end{itemize}
To conclude the proof, we verify that $\tilde{V}$ is almost minimizing in small annuli. If $q\in M$ belongs to $U$, then $\tilde{V}$ is almost minimizing in annuli centered at $q$ and contained in $U$. If $q \in M\setminus U$, choose $r^{\prime}_q>0$ such that $r^{\prime}_q\leq r_q$ and $dist_{M}(q, K)> r^{\prime}_q$. Since $\tilde{V}=V$ in $G_n(\R^N\setminus K)$, we conclude that $\tilde{V}$ is also almost minimizing in annuli with outer radii as most $r^{\prime}_q$. Observe that we can assume $r^{\prime}_p = r_p$. 
\end{proof}


We are now ready to present the proof of Theorem \ref{teorema-B}.

\begin{proof}[Proof of Theorem \ref{teorema-B}] Let $V\in \V_n(M)$ be the varifold obtained on Theorem \ref{comb-arg}. We already know that it satisfies the required properties at points in $M\setminus (\partial M)$. In order to prove boundary regularity and the other claims about the components of $V$ that intersect $\gamma$, it is enough to show that any varifold tangent $C \in \text{Var Tan}(V,p)$, for $p\in \gamma\cap spt(||V||)$, is an $n$-dimensional half-space. The result will follow from the theory in Allard's paper \cite{al-2}. 

Fix a varifold tangent $C$ as in the previous paragraph. Let $r_1, r_2, \ldots$, be a sequence of positive numbers decreasing to zero such that
\begin{equation}
C = \lim_{j\rightarrow \infty} (\eta_{p,r_j})_{\#} V.
\end{equation}
For large $j\in \N$, $V$ is almost minimizing in $U_j = A(p,2^{-1}r_j, 3r_j)$. Let $V_j$ be a replacement for $V$ in $K_j = clos(A(p,r_j, 2r_j))$. It follows from properties (i) and (iii) of Theorem \ref{replacement-properties}, and from the usual compactness result for varifolds, that, up to a subsequence, there exists
\begin{equation}\label{def-c-bar}
\overline{C} = \lim_{j\rightarrow \infty}(\eta_{p,r_j})_{\#} V_j,
\end{equation}
as varifolds in $\R^N$. Observe also that $\overline{C}=C$ in the Grassmannian over the complement in $\R^N$ of the closure of $A(O,1,2)$. We use $A(O,1,2)$ to denote the Euclidean annulus of radii $1$ and $2$, and centered at the origin $O\in \R^N$.

We claim that $\overline{C}$ is an integral varifold in $\R^N$, which is stationary in $(T_pM)\setminus (T_p\gamma)$. In order to verify this, we start by observing that $V_j$ has properties similar to those satisfied by $V$, as seen in Proposition \ref{replacement-further-prop}. In particular, $V_j$ is integral. Then, the varifolds $(\eta_{p,r_j})_{\#} V_j$ are integral and belong to $\V_s(M_j, \gamma_j)$ in the sense of definition \ref{stationary}, for $M_j = \eta_{p,r_j}(M)$ and $\gamma_j = \eta_{p,r_j}(\gamma)$. Since the submanifolds $M_j$ smoothly converge to the flat $T^+_pM$, the stationarity of $\overline{C}$, and its integrality away from $T_p\gamma$ follow from the compactness theorem of Allard \cite{al-3}. The full integrality follows from considerations similar to those in the proof of Lemma 6.4 of \cite{dl-r}. The first variation of $\overline{C}$ is non-negative with respect to compactly supported vector fields in $T^+_pM$ that are non-exterior at $\partial (T^+_pM)$, and vanish along $T_p\gamma$.

Property (iv) of Theorem \ref{replacement-properties} gives us that the restrictions
\begin{equation}\label{eq-replacement-j}
V_j^{\prime} = \big( (\eta_{p,r_j})_{\#} V_j \big) \llcorner G_n(A(O,1,2))
\end{equation}
are supported in stable minimal hypersurfaces in $M_j\cap A(O,1,2)$, which satisfy the required regularity and boundary assumptions of Theorem \ref{thm7.4-dr}. As observed in section 7.4 of \cite{dl-r}, the compactness theorem can still be applied for varying ambient spaces such as the $M_j$. The hypersurfaces in  (\ref{eq-replacement-j}) have uniformly bounded masses and are contained in compact sets which meet $\gamma_j$ at an opening angle $\theta_j$ uniformly away from $\pi/2$. The last fact holds because $V_j \in \V_s(M, \gamma)$, and Lemma \ref{lemma8.1-dr} applied for the choices $U_1=U_2=M$. 

It follows from expression (\ref{def-c-bar}) that $V_j^{\prime}$ converges to $\overline{C} \llcorner G_n(A(O,1,2))$ as varifolds in $A(O,1,2)$. Therefore, the compactness theorem implies that $\overline{C} \llcorner G_n(A(O,1,2))$ is supported on a stable minimal hypersurface $\Gamma$ whose singular set $sing(\Gamma) = \overline{\Gamma}\setminus \Gamma$ has Hausdorff dimension at most $n-7$, 
\begin{equation*}
sing(\Gamma)\cap \partial(T^+_pM) = \varnothing, \text{ and } \partial \Gamma\cap A(O,1,2) = T_p\gamma\cap A(O,1,2).
\end{equation*}
In particular, only one component of $\Gamma$ intersects $\partial (T^+_pM)$, and this component has multiplicity one in $\overline{C} \llcorner G_n(A(O,1,2))$.
 
Let $C_1$ be the the integral varifold obtained as the sum of $C$ and its reflection with respect to $T_p\gamma$, as defined in sections 2 and 3.2 of \cite{al-2}. In our setting, the reflection is given by $\theta(y) = y_{\gamma} - y^{\perp}$, for every $y \in T_pM$, where $y_{\gamma}$ is the orthogonal projection of $y$ to $T_p\gamma$, and $y^{\perp} = y - y_{\gamma}$. Thus, $C_1 = C + \theta_{\#}C$. Similarly, define $\overline{C}_1 = \overline{C}+ \theta_{\#}\overline{C}$. It follows from the reflection principle, section 3.2 of \cite{al-2}, that the integral varifolds $C_1$ and $\overline{C}_1$ are stationary in $T_pM$. Moreover, it is known, from Proposition \ref{k-half-spaces}, that $C_1$ is a cone. We also have that $C_1=\overline{C}_1$ in $G_n(B^{n+1}_1(0))$, and that
\begin{equation}
||C_1||(B^{n+1}_3(O)) = ||\overline{C}_1||(B^{n+1}_3(O)),
\end{equation} 
where $B^{n+1}_r(O)\subset T_pM$ denotes the open ball of radius $r$, centered at the origin. This implies that $C_1= \overline{C}_1$, and then, $C= \overline{C}$, see 2.4(6)(f) of \cite{pitts}.

Finally, since $\overline{C} \llcorner G_n(A(O,1,2))$ is a smooth minimal hypersurface with multiplicity one near $T_p\gamma$, and $C$ is a sum of half-spaces, as in the statement of Proposition \ref{k-half-spaces}, we conclude that $C$ is a half-space that contains $T_p\gamma$.
\end{proof}

\section{Proof of Theorem \ref{teorema-A}}\label{sect-teoremaA}

In this section we prove our main theorem. In the argument, we use the equivalence between singular homology and the homology of a chain complex of integral currents, see Theorem 5.11 of \cite{FF}.

\begin{proof}[Proof of Theorem \ref{teorema-A}]

Let $\Gamma_1$ and $\Gamma_2$ be the hypersurfaces in the statement of the theorem. We also use $\Gamma_1$ and $\Gamma_2$ to denote the integral $n$-currents associated with those. Our proof has three steps, similar to those in the proof of the analog result of \cite{dl-r}, Corollary 1.9 of that paper. The first step is the construction of a sweepout connecting $\Gamma_1$ and $\Gamma_2$. The goal of the second step is to show that the homotopy class of the sweepout obtained in step 1 is non-trivial. The third, and last, step is the application of the mountain pass result, which, in our case, is Theorem \ref{teorema-B}.

Let us start with the construction of the sweepout. Since $\partial (\Gamma_1-\Gamma_2)=0$, we can look at the homology class of $\Gamma_1-\Gamma_2$ in either $H_n(M)$, or $H_n(M, \gamma)$. Observe also that $[\Gamma_1-\Gamma_2]=0$ in $H_n(M)$ if and only if it is zero in $H_n(M, \gamma)$. Each of these assumptions are equivalent to the existence of $A\in I_{n+1}(M)$ such that $\partial A = \Gamma_1-\Gamma_2$. Since $\Gamma_1$ and $\Gamma_1$ are homologous in $M$, then $[\Gamma_1-\Gamma_2]=0$ in $H_n(M, \gamma)$ and there exists $A$ as above.

Intuitively, we can use Almgren's isomorphism, see section 3 of \cite{alm1},
\begin{equation*}
\pi_0 (\Z_n(M, \gamma), 0) \simeq H_n(M, \gamma),
\end{equation*}
to obtain a path in $\Z_n(M, \gamma)$ joining $\Gamma_1$ and $\Gamma_2$. Since we work with discrete sweepouts that are fine in $\M$, there is an obvious technical issue with simply using that path. In the next paragraph we explain how we overcome this.

Let $A\in I_{n+1}(M)$ be as above. Applying Lemma 6.1 of \cite{montezuma} to this current, and adding $\Gamma_2$ to the map obtained in this way, we conclude that there exists $\phi : [0,1]\rightarrow \Z_n(M, \gamma)$ with the following properties:
\begin{itemize}
\item $\phi(0) = \Gamma_1$ and $\phi(1) = \Gamma_2$,

\item $\phi$ is $\Flat$-continuous,

\item $\sup \{\M(\phi(t)) : t\in [0,1]\}< \infty$, and

\item $\limsup_{r\rightarrow 0} \m(\phi,r) = 0$, in the sense of item (iii) of section \ref{Sec.discretization}.
\end{itemize}
Therefore, we can apply Theorem \ref{discretization} to this map and obtain an $(1,\M)$-homotopy sequence of maps into $(Z_n(M,\gamma), P_0)$, for $P_0 : \partial I^1 \rightarrow \Z_n(M, \gamma)$ given by $P_0(0) = \Gamma_1$ and $P_0(1) = \Gamma_2$. Indeed, observe that the sequence $\{\varphi_i\}$ provided by that theorem has the properties required by definition \ref{sweepout}; properties (a), (b), (d), and (e) of Theorem \ref{discretization} imply that $\varphi_i$ and $\varphi_{i+1}$ are homotopic in $(\Z_n(M, \gamma), P_0)$ with $\M$-fineness $\delta_i$, and property (c) and the mass bound of $\phi$ imply that $\M(\varphi_i(x))$ are also uniformly bounded.

Let $\Pi \in \Pi_1^{\#}(Z_n(M,\gamma), P_0)$ be the homotopy class generated by $\{\varphi_i\}$. We claim that this class in non-trivial in the sense of Theorem \ref{teorema-B}, i.e.,
\begin{equation}\label{non-triv-1}
\LL (\Pi) > max \{\M(\Gamma_1), \M(\Gamma_2)\}.
\end{equation}
Since our maps are fine with respect to the mass norm, which is finer that the topology of the $\Flat$ norm, and $\Flat(\Gamma_1, \Gamma_2)>0$, the fact expressed in (\ref{non-triv-1}) follows from Lemma 11.2 of \cite{dl-r}. See also the main result of \cite{IM}.

Finally, apply Theorem \ref{teorema-B} to $\Pi$. It gives us a varifold $V = \sum_{i=1}^k m_i \overline{\Sigma_i}$, where $\Sigma_i$ are embedded minimal hypersurfaces and $m_i$ positive integers. We consider two cases. If some of the $\Sigma_i$ is closed, we observe that $\Sigma = spt(||V||)$ satisfies all the desired properties. It is distinct from $\Gamma_1$ and $\Gamma_2$, because we are assuming these do not have closed components. If all $\Sigma_i$ have non-empty boundary, then $m_i=1$, for all $i$, and we conclude that
\begin{equation*}
\H^n(\Sigma) = ||V||(M) = \LL (\Pi) > max\{\H^n(\Gamma_1), \H^n(\Gamma_2)\}.
\end{equation*}
In particular, $\Sigma$ is distinct from $\Gamma_1$ and $\Gamma_2$. It could happen that this $\Sigma$ is a combination of connected components of $\Gamma_1$ and $\Gamma_2$. We explain how this is ruled out in the next paragraph by exploring further one of the ideas in the proof of Corollary 1.9 of \cite{dl-r}.

Consider all embedded hypersurfaces with boundary $\gamma$ which are made of a combination of connected components of $\Gamma_1$ and $\Gamma_2$. These are all strictly stable minimal hypersurfaces. Pick the hypersurface $\tilde{\Gamma}_1$ in this class with maximal $\H^n$-measure, and let $\tilde{\Gamma}_2$ be the union of the connected components of $\Gamma_1$ and $\Gamma_2$ that are not in $\tilde{\Gamma}_1$. More precisely, if
\begin{equation}\label{eq-components}
\Gamma_1 = \Gamma(1,1) + \Gamma(1,2) \text{ and } \Gamma_2 = \Gamma(2,1) + \Gamma(2,2),
\end{equation}
where $\Gamma(1,1)$ and $\Gamma(2,1)$ are the connected components of $\Gamma_1$ and $\Gamma_2$ that appear in $\tilde{\Gamma}_1$, respectively, then
\begin{equation}\label{eq-components-2}
\tilde{\Gamma}_1 = \Gamma(1,1) - \Gamma(2,1) \text{ and } \tilde{\Gamma}_2 = \Gamma(2,2) - \Gamma(1,2).
\end{equation}
In expressions (\ref{eq-components}) and (\ref{eq-components-2}), the orientations are also considered. Observe that $\tilde{\Gamma}_1-\tilde{\Gamma}_2 = \Gamma_1 - \Gamma_2$, which implies that $\partial \tilde{\Gamma}_2$ is also $\gamma$, and that the hypersurfaces $\tilde{\Gamma}_i$ are homologous. It can also be verified that these are distinct. Then, we can construct a discrete sweepout joining the $\tilde{\Gamma}_i$ as before. It also follows immediately, from similar considerations, that the homotopy class $\Pi$ generated by this sweepout satisfies $\LL (\Pi) > \H^n(\tilde{\Gamma}_1)$. Since the hypersurfaces in our statement could intersect, we do not know whether or not $\tilde{\Gamma}_2$ is embedded. But, using that $\Gamma_1$ and $\Gamma_2$ are also competitors for $\tilde{\Gamma}_1$, and that $\H^n(\tilde{\Gamma}_1) + \H^n(\tilde{\Gamma}_2) = \H^n(\Gamma_1) + \H^n(\Gamma_2)$, we are able to check that 
\begin{equation*}
max \{\H^n(\tilde{\Gamma}_1), \H^n(\tilde{\Gamma}_2)\} = \H^n(\tilde{\Gamma}_1) < \LL (\Pi) .
\end{equation*} 
Thus, $\Pi$ is non-trivial and the theorem is proved.

\end{proof}

\bibliographystyle{amsbook}

\begin{thebibliography}{99}

\bibitem{al-1}
W. Allard, \textit{On boundary regularity for Plateau's problem.} Bull. Amer. Math. Soc. 75 1969 522--523.

\bibitem{al-3}
W. Allard, \textit{On the first variation of a varifold.}  Ann. of Math. (2) 95 (1972), 417--491.

\bibitem{al-2}
W. Allard, \textit{On the first variation of a varifold: boundary behavior.} Ann. of Math. (2) 101 (1975), 418--446.

\bibitem{alm1}
F. Almgren, \textit{The homotopy groups of the integral cycle groups.} Topology (1962), 257--299.

\bibitem{C-DL}
T. Colding and C. De Lellis, \textit{The min-max construction of minimal surfaces,} Surveys in  Differential Geometry VIII , International Press, (2003),   75--107.


\bibitem{dl-r}
C. De Lellis and J. Ramic,
\textit{Min-max theory for minimal hypersurfaces with boundary.} preprint, arXiv:1611.00926 [math.AP].

\bibitem{dl-t}
C. De Lellis and D. Tasnady
\textit{The existence of embedded minimal hypersurfaces.}  J. Differential Geom. 95 (2013), no. 3, 355--388.

\bibitem{DS}
F. Duzaar and K. Steffen,
\textit{Optimal interior and boundary regularity for almost minimizers to elliptic variational integrals.}  J. Reine Angew. Math. 546 (2002), 73--138.

\bibitem{FF}
H. Federer and W. Fleming,
\textit{Normal and integral currents.} Ann. of Math. (2) 72 1960 458--520. 

\bibitem{GLZ}
Q. Guang, M. Li, and X. Zhou,
\textit{Curvature estimates for stable free boundary minimal hypersurfaces.} preprint, arXiv:1611.02605 [math.DG].

\bibitem{IM}
D. Inauen and A. Marchese,
\textit{Quantitative minimality of strictly stable minimal submanifolds in a small flat neighborhood.} preprint, arXiv:1709.02652 [math.AP].

\bibitem{JS}
J. Jost and M. Struwe,
\textit{Morse-Conley theory for minimal surfaces of varying topological type.} Invent. Math. 102 (1990), no. 3, 465--499. 

\bibitem{Li-Zhou}
M. Li, and X. Zhou,
\textit{Min-max theory for free-boundary minimal hypersurfaces I - Regularity Theory.} preprint, 	arXiv:1611.02612 [math.DG].

\bibitem{MN-willmore}
F. Marques and A. Neves,
\textit{Min-max theory and the Willmore conjecture.} Ann. of Math. (2) 179 (2014), no. 2, 683--782.

\bibitem{MN-index}
F. Marques and A. Neves,
\textit{Morse index and multiplicity of min-max minimal hypersurfaces.} Camb. J. Math. 4 (2016), no. 4, 463--511.

\bibitem{montezuma}
R. Montezuma,
\textit{Min-max minimal hypersurfaces in non-compact manifolds.} J. Differential Geom. 103 (2016), no. 3, 475--519.

\bibitem{M+T}
M. Morse and C. Tompkins,
\textit{The existence of minimal surfaces of general critical types.} Ann. of Math. (2) 40 (1939), no. 2, 443--472. 

\bibitem{Sh}
M. Shiffman,
\textit{The Plateau problem for non-relative minima.} Ann. of Math. (2) 40, (1939). 834--854. 

\bibitem{simon}
L. Simon, \textit{Lectures on geometric measure theory,}
Proceedings of the Centre for Mathematical Analysis, Australian National University,  Canberra, (1983). vii+272 pp. 


\bibitem{pitts}
J. Pitts, \textit{Existence and regularity of minimal surfaces on Riemannian manifolds,} Mathematical Notes 27, Princeton University Press, Princeton, (1981).

\bibitem{SSY}
R. Schoen, L. Simon and S. T. Yau,
\textit{Curvature estimates for minimal hypersurfaces.} Acta Math. 134 (1975), no. 3-4, 275--288.

\bibitem{SS}
R. Schoen and L. Simon,
\textit{Regularity of stable minimal hypersurfaces.} Comm. Pure Appl. Math. 34 (1981), no. 6, 741--797.


\end{thebibliography}

\end{document}